\documentclass[11pt]{amsart}

\usepackage{amsmath}
\usepackage{amssymb}

\allowdisplaybreaks                     
\theoremstyle{plain}
\newtheorem{teo}{Theorem}[section]
\newtheorem*{theorem*}{Theorem}
\newtheorem{lem}[teo]{Lemma}
\newtheorem{coro}[teo]{Corollary}
\newtheorem{prop}[teo]{Proposition}

\theoremstyle{remark}
\newtheorem{remark}[teo]{Remark}
\numberwithin{equation}{section}
\newcommand{\quadro}{\hfill\vrule height .9ex width .8ex depth -.1ex}

\newcommand{\NN}{\Bbb{N}}

\newcommand{\RR}{\Bbb{R}}
\newcommand{\CC}{\Bbb{C}}

\newcommand{\DR}{Damek--Ricci }

\newcommand{\n}{\mathfrak{n}}
\newcommand{\ag}{\mathfrak{a}}

\newcommand{\vg}{\mathfrak{v}}

\newcommand{\zg}{\mathfrak{z}}
\newcommand{\s}{\mathfrak{s}}
\newcommand{\LB}{\Delta} 
\newcommand{\D}{L} 
\newcommand{\LQ}{\Delta_Q} 
\newcommand{\kD}{k_{\psi(L)}}
\newcommand{\kQ}{k_{\psi(\Delta_Q)}}
\newcommand{\du}{\delta^{1/2}} 

\newcommand{\di}{\,{\rm{d}}}
\newcommand{\dir}{\,{\rm{d}}\rho}  
\newcommand{\dil}{\,{\rm{d}}\lambda}  

\newcommand{\nep}{{\rm{e}}} 

\newcommand{\supp}{{\rm{supp }\,}}

\begin{document}
\title[Wave equation and multiplier estimates on Damek--Ricci spaces]{\bf{Wave equation and multiplier estimates\\ on Damek--Ricci spaces}}



\thanks{{\bf Acknowledgement.}  Work partially supported by the European Commission via the Network HARP, ``Harmonic analysis and related problems''. }

\author[D. M\"uller, M. Vallarino]
{Detlef M\"uller and Maria Vallarino}

\address{Detlef M\"uller:
Matematisches Seminar
\\ Christian-Albrechts Universit\"at  \\
Ludewig-Meyn-Strasse 4
\\ D-24098 Kiel\\ Germany -- mueller@math.uni-kiel.de}
\address{Maria Vallarino:
Laboratoire MAPMO\\
 Universit\'e d'Orl\'eans, UFR Sciences\\
B\^atiment de Math\'ematique-Route de Char\-tres\\
B.P. 6759\\ 45067 Orl\'eans cedex 2,   France --  maria.vallarino@unimib.it}

\begin{abstract}
Let $S$ be a \DR space and $\D$ be a distinguished left invariant Laplacian on $S$. We prove pointwise estimates for the convolution kernels of spectrally localized wave operators of the form
$$\nep^{it\sqrt{\D}}\,\psi\big(\sqrt{\D}/{\lambda}\big)$$
for arbitrary time $t$ and arbitrary $\lambda >0$, where $\psi$ is a smooth bump function supported in $[-2,2]$ if $\lambda <1$ and supported in $[1,2]$ if $\lambda \geq 1$. This generalizes previous results in \cite{MT}. We also prove pointwise estimates for the gradient of these convolution kernels. As  a corollary, we reprove basic multiplier estimates from \cite{HS} and \cite{V1} and derive Sobolev estimates for the solutions to the wave equation associated to $\D$.
\end{abstract}

\maketitle

\section{Introduction}
Let $\frak{n}=\frak{v}\oplus\frak{z}$ be an $H$-type algebra and let $N$ be the connected and simply connected Lie group associated to $\frak{n}$ (see Section 
\ref{DRsec} for the details). By  $S$ we denote the one-dimensional solvable extension  of $N$ obtained by making $A=\Bbb{R}^+$ act 
on $N$ by homogeneous dilations $\delta_a$. We choose  $H$ in the Lie algebra  $\frak{a}$ of $A$ so that $\exp(tH)=e^t, t\in \RR,$ and extend the inner product on the Lie algebra 
 $\frak{n}$ of $N$ to the Lie algebra $\frak{s}=\frak{n}\oplus\frak{a}$ of $S$, by requiring $\frak{n}$
and $\frak{a}$ to be orthogonal and $H$ to be a unit vector. Let $d$ be the \emph{left invariant} Riemannian metric on $S$ 
which agrees with the inner product on $\frak{s}$ at the identity. 

By  $\lambda$ and $\rho$ we  denote the left and right invariant Haar measures on $S$, respectively. It is well known that both the 
right and the left Haar measures of geodesic balls are exponentially growing functions of the radius,
so that $S$ is a group of \emph{exponential growth}. 

The space $S$ is called a \emph{Damek--Ricci space}; these spaces were 
introduced by E.~Damek and F.~Ricci \cite{D1, D2, DR1, DR2},
and include all rank one symmetric spaces of the noncompact type.
Most of them are nonsymmetric harmonic manifolds,  
and provide counterexamples to the Lichnerowicz conjecture. The geometry of these extensions was studied by M.~Cowling, 
A.~H.~Dooley, A.~Kor\'anyi and Ricci in \cite{CDKR1, CDKR2}.

The simplest example of Damek--Ricci spaces is given by the so called $ax+b$-groups, which can be thought as the harmonic extensions of $N=\mathbb R^d$. 

Let $\{X_0,\ldots,X_{n-1}\}$ be an orthonormal basis of the Lie algebra $\frak{s}$ such that 
$X_0=H$, the elements $X_1,\ldots,X_{m_{\frak{v}}}$ form an orthonormal basis of $\frak{v}$ 
and $X_{m_{\frak{v}}+1},\ldots,X_{n-1}$ form an orthonormal basis of $\frak{z}$. As usually, we shall identify an element $X\in\s$ with the corresponding left invariant differential operator on $S$  given by the Lie derivative $Xf(g)= \frac d{dt}f(g\exp tX)\Big|_{t=0}.$ 
Viewing $X_0,X_1,\ldots,X_{n-1}$ in this way as  left invariant vector fields, we  define the left invariant  \emph{Laplacian} $L$ by  
$$
\D= -\sum_{i=0}^{n-1} X_i^2\,.
$$
Then $L$  is essentially selfadjoint on $C^{\infty}_c(S)\subset L^2(\rho)$ 
and its $L^2(\rho)$--spectrum is  given by $[0,\infty)$. 

For any  Borel measurable bounded multiplier $m$ on $[0,\infty),$ we can thus define  the operator
$m(\D)
$
on $L^2(\rho)$ by means of spectral calculus.
Then  $m(\D)$ is  left invariant  too, so that, as a consequence of the Schwartz kernel theorem, there exists a unique distribution $k$ on $S$ such that
$$
m(\D)f=f\ast k\qquad \forall f\in \mathcal D(S)\,.
$$

Several  authors have investigated the $L^p$--functional calculus for the Laplacian $L$, i.e.,  they studied sufficient and necessary conditions on a multiplier $m$ such that the operator $m(\D)$ extends from  $L^2(\rho)\cap L^p(\rho)$ to an  $L^p(\rho)$- bounded operator, for  a given  $p$ in $(1,\infty)$ \cite{H, CGHM, A2, HS, V1} (in some of these papers, the authors work with  a right invariant Laplacian $\D_r,$  sometimes on the class of solvable groups arising  in  the
Iwasawa decomposition of noncompact connected semisimple Lie groups 
of finite centre, but one can easily pass to our left-invariant Laplacian $L$ by means of group inversion).
\smallskip

The main purpose of this article is to prove pointwise estimates for 
the convolution kernels of spectrally localized multiplier operators 
of the form
$$
e^{it\sqrt \D}\,\psi(\sqrt \D/\lambda),
$$
for arbitrary time $t$ and $\lambda>0$, where $\psi$ is a bump function 
supported in $[-2,2]$, if $\lambda\leq 1$, and in $[1,2]$, if $\lambda >1$. 
Such estimates were proved in \cite{MT} in the case of $ax+b$-groups. 
In this paper we generalize the results in \cite{MT} to Damek--Ricci spaces, 
and we also find pointwise estimates of the gradient of these kernels. 

We shall use these estimates to give a new proof of basic multiplier results 
in \cite{HS, V1}, which is based entirely on the wave equation.

Moreover, as a corollary we obtain $L^p$--estimates for Fourier multiplier 
operators of the form $m(\sqrt{\D})\,\cos(t\,\sqrt{\D})$ and 
$m(\sqrt{\D})\,\frac{\sin(t\,\sqrt{\D})}{\sqrt{\D}}$, 
where $m$ is a suitable symbol, and we 
derive Sobolev estimates for  solutions of the wave equation associated 
with $\D$ on Damek--Ricci spaces. 

This problem was first studied in the euclidean setting in \cite{M, P}. 
The problem of the regularity in space for fixed time of the wave equation 
associated with the Laplace--Beltrami operator on a 
noncompact symmetric space of arbitrary rank was studied in \cite{GM, CGM}. 
In the case of noncompact symmetric spaces of rank one A.~Ionescu \cite{I}
estimated the $L^p$--norm of the Fourier integral operators 
for variable time and derived Sobolev estimates for the solution of the 
wave equation associated with a shifted Laplace--Beltrami operator. Since this is the counterpart of our result in the same setting for a different Laplacian 
we shall discuss it in more details at the end of our paper.

Our paper is organized as follows. In Section 2 we recall the definition of $H$-type groups and Damek--Ricci spaces and we summarize some results about 
spherical analysis on such spaces. In Section 3 we recall some properties 
of the Laplacian $L$ and derive a formula for convolution kernels of multipliers of $L$. In Section 4 and 5 we prove pointwise estimates for the spectrally localized wave propagator and for its gradient, respectively. We deduce an $L^1$-estimate for these kernels; then we apply it to reprove a basic multiplier estimate from \cite{HS} and \cite{V1}. In Section 6 we derive Sobolev estimates for the solutions to the wave equation associated to $\D$.

We shall apply the ''variable constant'' convention in this paper, according to which  $C$ will usually denote a positive, finite constant which may vary from line to line and may depend on parameters according to the context. Given two quantities $f$ and $g$, by $f\lesssim g$ we mean that there exists a constant $C$ such that $f\leq C\,g$ and by $f \asymp g$ we mean that there exist constants $C_1,\,C_2$ such that $C_1\,g\leq f\leq C_2\,g$.

\section{Damek--Ricci spaces}\label{DRsec}
In this section we recall the definition of $H$-type groups, which had been introduced by Kaplan \cite{K}, describe their harmonic extensions and recall the main results of spherical analysis on these extensions. For the details see \cite{ADY, CDKR1, CDKR2}.
\smallskip

Let $\n$ be a two-step nilpotent Lie algebra equipped  with an inner product $\langle\cdot,\cdot\rangle$ and denote by $|\cdot |$ the corre\-spon\-ding norm. Let $\vg$ and $\zg$ be complementary orthogonal subspaces of $\n$ such that $[\n,\zg ]=\{0\}$ and $[\n,\n]\subseteq \zg$.  


The algebra $\n$ is of {\it $H$-type} if for every  $z$ in $\zg$ the map $J_z:\vg\to \vg$ defined by
$$\langle J_zv,v'\rangle\,=\,\langle z,[v,v']\rangle\qquad\forall v, v'\in \vg$$
satisfies the condition
$$
|J_zX|=|Z|\,|X|\qquad \forall X\in \vg\quad \forall Z\in\zg.
$$ 
The connected and simply connected Lie group $N$ associated to $\n$ is called an {\it $H$-type group.} We identify $N$ with its Lie algebra $\n$ via the exponential map
\begin{align*}
\vg\times\zg &\to N\\
(v,z)&\mapsto \exp(v+z)\,.
\end{align*}
The product law in $N$ is given by
$$(v,z)(v',z')=\big(v+v',z+z'+({1}/{2})\,[v,v']\big)\qquad\forall v,\,v'\in \vg\quad\forall z,\,z'\in\zg\,.$$
The group $N$ is  two-step nilpotent, hence unimodular, with Haar measure $\di v \di z$.
We define the following  dilations on $N$:
\begin{align*}
\delta_a(v,z)&=(a^{1/2}v,az)\qquad\forall (v,z)\in N \quad\forall a\in\RR^+\,.
\end{align*}
These are automorphisms, so that $N$ is an homogeneous group. A homogeneous  norm is given by
$$\mathcal N (v,z)=\left(\frac{|v|^4}{16}+|z|^2\right)^{1/4}\qquad\forall (v,z)\in N\,.$$
Note that $\mathcal N\big(\delta_a(v,z)\big)=a^{1/2}\mathcal N(v,z),$ so that $\mathcal N$ is homogeneous of degree one with respect to the modified dilation group $\{\delta_{r^2}\}_{r>0}.$  Set $Q={(m_{\vg}+2m_{\zg})}/{2}\,$, where $m_{\vg}$ and $m_{\zg}$ denote the dimensions of $\vg$ and $\zg$, respectively. 

Let $H$ be the element of $\ag$ such that ${\rm{ad}}(H)v=\frac{1}{2}v$, if $v\in \vg$, and ${\rm{ad}}(H)z=z$, if $z\in \zg$. We extend the inner
product on $\n$ to the algebra $\s=\n\oplus\ag$, by requiring $\n$
and $\ag$ to be orthogonal and $H$ to be unitary. The algebra $\s$
is a solvable Lie algebra. The corresponding Lie group $S$ is the semi-direct extension $S=N\rtimes A$, where $A=\mathbb{R}^+$ acts on $N$ by the above dilations. 

The map
\begin{align*}
\vg\times\zg\times\RR^+ &\to S\\
(v,z,a)&\mapsto \exp(v+z)\exp(\log a \,H)
\end{align*}
gives global coordinates on $S$. The product in $S$ is then given by the rule
$$(v,z,a)(v',z',a')=\big(v+a^{1/2}v',z+a\,z'+({1}/{2})\,a^{1/2}[v,v'],a\,a'\big )$$
for all $(v,z,a),\,(v',z',a')$ in $S$. Let $e$ be the identity of the group $S$. We shall denote by $n=m_{\vg}+m_{\zg}+1$ the dimension of $S$. The
group $S$ is nonunimodular: the right and left Haar measures on
$S$ are  given by $\dir(v,z,a)=a^{-1}\di v\di z\di a\,$ and
$\dil(v,z,a)=a^{-(Q+1)}\di v\di z\di a\,,$ respectively. In particular,
$$
\dil=\delta\,\dir,
$$
where the modular function is given by 
 $\delta(v,z,a)=a^{-Q}$. 

We endow $S$ with the left invariant Riemannian metric which agrees with the inner product on $\s$ at the identity. Let $d$ denote the distance induced by this Riemannian structure. The Riemannian manifold $(S,d)$ is then called a \emph{\DR space} or {\it harmonic $NA$ group.} 

 It is well known \cite[formula (2.18)]{ADY} that 
\begin{align}\label{distanza}
\cosh ^2\left(\frac{d\big((v,z,a),e\big)}{2}\right)=\left(\frac{a^{1/2}+a^{-1/2}}{2}+\frac{1}{8}\,a^{-1/2}|v|^2\right)^2+\frac{1}{4}\,a^{-1}|z|^2\qquad \forall (v,z,a)\in S\,.
\end{align}
\medskip
We shall later use the notation
\begin{equation}\label{R}
R(x)=d(x,e), \quad x\in S.
\end{equation}

Let us remark at this point that it is  natural to include also the ''degenerate'' cases where $N$ is abelian into the definition of  H-type  group respectively of Damek-Ricci space, since then all symmetric spaces of noncompact type and rank one, including real hyperbolic spaces,  are special cases of these spaces, when considered as Riemannian manifolds (cf. \cite{CDKR2}). Indeed, all of our previous definitions and subsequent arguments will apply as well when $N$ is abelian. Only when $N$ is one-dimensional, some estimates will have to be modified. However, since the case where $N$ is abelian had essentially already been dealt with in \cite{MT}, we shall restrict ourselves in this paper to the case where $N$ is non-abelian. 

 \medskip
We denote by $B\big((v_0,z_0,a_0),r\big)$ the ball in $S$ centred at $(v_0,z_0,a_0)$ of radius $r$. In particular let $B_r$ denote the ball of center $e$ and radius $r$; note that \cite[formula (1.18)]{ADY} 
\begin{align}\nonumber
\rho\big(B_r\big)\asymp \begin{cases} 
r^n&{\rm{if~}}r<1\\
\nep^{Qr}&{\rm{if~}}r\geq 1\,.
\end{cases}
\end{align}
This shows in particular that $S$, equipped with the right Haar measure $\rho$, is a group of exponential growth.

A {\it radial} function on $S$ is a function that depends only on the distance from the identity. If $f$ is radial, then \cite[formula (1.16)]{ADY}
\begin{equation}\label{intsin}
\int_Sf\dil=\int_0^{\infty}f(r)\,A(r)\di r\,,
\end{equation}
where 
\begin{align*}
A(r)&=2^{m_{\vg}+2m_{\zg}}\sinh ^{m_{\vg}+m_{\zg}}\left(\frac{r}{2}\right)\cosh ^{m_{\zg}}\left(\frac{r}{2}\right)\qquad\forall r\in\RR^+\,.
\end{align*}
One easily checks that
\begin{align}\label{pesoA}
A(r)&\lesssim \left(\frac{r}{1+r}\right)^{n-1}\nep^{Qr}\qquad\forall ~r\in\RR^+\,.
\end{align}
A radial function $\phi$ is {\it spherical} if it is an eigenfunction of the Laplace-Beltrami operator $\Delta$ (associated to $d$) and $\phi(e)=1$. Let $\phi_s$, for $s\in\CC$, be the spherical function with eigenvalue $s ^2+Q^2/4$, as in \cite[formula (2.6)]{ADY}.

In \cite[Lemma 1]{A1} it is shown that
\begin{equation}\label{stimafi0}
\phi_0(r)\lesssim (1+r)\,\nep^{-Qr/2}\qquad\forall~ r\in\RR^+\,.
\end{equation}
We shall use the following integration formula on $S$, whose proof is reminiscent of \cite[Lemma 1.3]{CGHM} and \cite[Lemma 3]{A1}:
\begin{lem}\label{intduf}
For every radial function $f$ in $C_c^{\infty}(S)$
\begin{align*}
\int_S\delta^{1/2}f \,{\rm{d}} \rho &=\int_0^{\infty}\phi_0(r)\,f(r)\,A(r)\,{\rm{d}} r \\
&=\int_0^{\infty}f(r)\,J(r)\,{\rm{d}} r\,,
\end{align*}
where 
$$J(r)\lesssim  \begin{cases}
r^{n-1}&{\rm{if~}}r<1\\
r\,\nep^{Qr/2}&{\rm{if~}}r\geq 1\,.
\end{cases}$$
\end{lem}
The {\it spherical Fourier transform} of an integrable radial function $f$ on $S$ is defined by the formula
$$\mathcal H f(s)=\int_S \phi _{s}\,f\dil \,.$$
For ``nice'' radial functions $f$ on $S$ an inversion formula and a Plancherel formula hold:
$$f(x)=c_S\int_{0}^{\infty}\mathcal H f (s)\,\phi _{s}(x)\,| {\bf{c}} (s) |^{-2} \di s\qquad \forall x\in S\,,$$
and
$$\int_S|f|^2\dil =c_S\int_0^{\infty}|\mathcal H f (s)|^2\,|{\bf{c}}(s)|^{-2}\di s\,,$$
where the constant $c_S$ depends only on $m_{\vg}$ and $m_{\zg}$ and ${\bf{c}}$ denotes the Harish-Chandra fun\-ction. 

Let $\mathcal A$ denote the Abel transform and let $\mathcal F$ denote the Fourier transform on the real line, defined by $\mathcal F g(s)=\int_{-\infty}^{+\infty}g(r)\,\nep^{-isr}\di r\,,$ for each integrable function $g$ on $\RR$. It is well known that $\mathcal H=\mathcal F\circ \mathcal A$, hence $\mathcal H^{-1}=\mathcal A^{-1}\circ \mathcal F^{-1}$. We shall use the inversion formula for the Abel transform \cite[formula (2.24)]{ADY}, which we now recall. Let $\mathcal D_1$ and $\mathcal D_2$ be the differential operators on the real line defined by
\begin{align*}
\mathcal D_{1}=\,-\frac{1}{\sinh r}\,\frac{\partial}{\partial r}\,,\qquad \mathcal D_{2}=\,-\frac{1}{\sinh(r/2)}\,\frac{\partial}{\partial r} \,.
\end{align*} 
If $m_{\zg}$ is even, then
\begin{equation}\label{inv1}
\mathcal A^{-1}f(r)=\tilde a_S^e\,\mathcal D_1^{m_{\zg}/2}\mathcal D_2^{m_{\vg}/2}f (r)\,,
\end{equation}
where $\tilde a_S^e=2^{-(2m_{\vg}+m_{\zg})/2}\pi^{-(m_{\vg}+m_{\zg})/2}$, while if $m_{\zg}$ is odd, then
\begin{align}\label{inv2}
\mathcal A^{-1}f(r)=\tilde a_S^o\int_r^{\infty}\mathcal D_1^{(m_{\zg}+1)/2}\mathcal D_2^{m_{\vg}/2}f (s)(\cosh s-\cosh r)^{-1/2}\sinh s \di s\,,
\end{align}
where $\tilde a_S^o= 2^{-(2m_{\vg}+m_{\zg})/2}\pi^{-n/2}.$ 

\section{The Laplacian $\D$}\label{Laplacian}

Let $\{X_0,\ldots,X_{n-1}\}$ be an orthonormal basis of the Lie algebra $\frak{s}$ such that 
$X_0=H$, the elements $X_1,\ldots,X_{m_{\frak{v}}}$ form an orthonormal basis of $\frak{v}$ 
and $X_{m_{\frak{v}}+1},\ldots,X_{n-1}$ form an orthonormal basis of $\frak{z}$. As before, we shall view the $X_i$ as left invariant vector fields on $S.$ In particular, if $i=0$, then
\begin{align*}
X_0f(v,z,a)&=a\,\partial _a f(v,z,a)\qquad \forall f\in C^{\infty}(S)\,,
\end{align*}
while  for  $i\neq 0$ the vector fields $X_i$ do not involve a derivative in the variable $a.$

These vector fields $X_i$, $i=0,...,n-1,$ form an orthonormal basis 
of the tangent space at every point of $S$, since the Riemannian metric 
is left invariant. If we denote by 
$\nabla$ the Riemannian gradient on $S$, it is easy to verify that for any 
function $f$ in $C^{\infty}(S)$
$$
\nabla f =\sum_i (X_if) \,X_i.
$$
This implies that the Riemannian norm of the gradient, which we denote by 
$\|\nabla f\|,$ is given by 
$$
\|\nabla f\|=\big( \sum_i |X_if|^2 \big)^{1/2}.
$$
Let $
L= -\sum_{i=0}^{n-1} X_i^2
$ be the left invariant  Laplacian defined in the Introduction.

The operator $\D$  has a special relationship with the (positive definite) {\it Laplace-Beltrami operator}  $\LB=-{\rm{div}}\circ{\rm{grad}}$. Indeed, let $\LQ$ denote the shifted operator $\LB-{Q^2}/{4}$; it is known that \cite[Proposition 2]{A1}
\begin{align}\label{relationship}
\delta^{-1/2}\,L\,\delta^{1/2}f=\LQ f\,,
\end{align}
for smooth  radial functions $f$ on $S$.\\
The spectra of $\LQ$ on $L^2(\lambda)$ and $L$ on $L^2(\rho)$ are  both $[0,+\infty)$. Let $E_{\LB _Q}$ and $E_{L}$ be the spectral resolution of the identity for which 
$$\LB _Q=\int_0^{+\infty}t\, \di E_{\LB _Q}(t)\qquad{\rm and}\qquad L=\int_0^{+\infty}t\, \di E_{L}(t)\,.$$
For each bounded Borel measurable function $\psi$ on $\RR^+$ the operators $\psi(\LB _Q)$ and $\psi(L)$, spectrally defined by  
$$\psi(\LB _Q)=\int_0^{+\infty}\psi(t) \di E_{\LB _Q}(t)\qquad{\rm and}\qquad \psi(L)=\int_0^{+\infty}\psi(t) \di E_{L}(t)\,,$$
are bounded on $L^2(\lambda)$ and $L^2(\rho)$ respectively. By (\ref{relationship}) and the spectral theorem, we see that 
$$\delta^{-1/2}\psi(L)\,\delta^{1/2}f=\psi(\LQ)f\,,$$ 
for smooth compactly supported radial functions $f$ on $S$.\\
Let $\kD$ and $\kQ$ denote the convolution kernels of $\psi(\D)$ and $\psi(\LQ)$ respectively; we have that 
$$\psi (\LQ) f=f\ast\kQ\qquad{\rm {and}}\qquad \psi(\D) f=f\ast\kD\qquad\forall f\in C^{\infty}_c(S)\,,$$
where $\ast$ denotes the convolution on $S$ defined by
\begin{align*}
f\ast g(x)&=\int_S f(y)g(y^{-1}x)\dil(y)=\int_Sf(xy)\,g(y^{-1})\dil (y)\\
&=\int_Sf(xy^{-1})\,g(y)\dir (y)\,,
\end{align*}
for all functions $f,g$ in $C_c(S)$ and $x$ in $S$.

The integral kernel of $\psi(\D)$ is the function defined on $(S,\dir)\times (S,\dir)$  by
\begin{equation}\label{integralkernel}
K_{\psi(\D)}(x,y)=k_{\psi(\D)}(y^{-1}x)\,\delta(y)\qquad \forall x,y\in S\,.
\end{equation}
\begin{prop}\label{relazionenuclei}
Let $\psi$ be a bounded measurable function on $\RR^+$. Then $\kQ$ is radial and $\kD=\du\,\kQ$. The spherical transform of $k_{\LQ}$ is
$$\mathcal H k_{\psi(\LQ)}(s)=\psi(s^2)\qquad\forall s \in\RR^+\,.$$
\end{prop}
\begin{proof}
See \cite{A1}, \cite{ADY}.
\end{proof}
By Proposition \ref{relazionenuclei} it follows that if $\psi$ is a function in $C_c(\RR)$, then the convolution kernel of $\psi(\D)$ is equal to
$$\kD(x)=(2\pi)^{-1}\du(x)\,\mathcal A^{-1}\Big(\int_{\RR}\psi(s^2)\,\nep^{isv}\di s\Big)\big(R(x)\big)\qquad \forall x\in S\,,$$
where $R(x)=d(x,e)$ is as in \eqref{R}.
 By using the expression of the inverse Abel transform (\ref{inv1}) and (\ref{inv2}) we deduce that if $m_{\zg}$ is even, then
\begin{align*}
\kD(x)&=a_S^e\,\du(x)\,\mathcal D_{1,v}^{m_{\zg}/2}\mathcal D_{2,v}^{m_{\vg}/2}\Big(\int_{\RR}\psi(s^2)\,\nep^{isv}\di s\Big)\big(R(x)\big)\\
&=a_S^e\,\du(x)\int_{\RR}\psi(s^2)\,\mathcal D_{1,v}^{m_{\zg}/2}\mathcal D_{2,v}^{m_{\vg}/2}(\nep^{isv})\big(R(x)\big)\di s\,,
\end{align*}
while if $m_{\zg}$ is odd, then
\begin{align*}
\kD(x)&=a_S^o\,\du(x)\int_{R(x)}^{\infty}\mathcal D_{1,v}^{(m_{\zg}+1)/2}\mathcal D_{2,v}^{m_{\vg}/2}\Big(\int_{\RR}\psi(s^2)\,\nep^{isv}\di s\Big)(v) 
\di\nu_{R(x)}(v)\\
&=a_S^o\,\du(x)\int_{\RR}\psi(s^2)\int_{R(x)}^{\infty}\mathcal D_{1,v}^{(m_{\zg}+1)/2}\mathcal D_{2,v}^{m_{\vg}/2}(\nep^{isv})(v)\di\nu_{R(x)}(v)\di s\,,
\end{align*}
where $a_S^e=(2\pi)^{-1}\tilde a_S^e,$  $a_S^o=(2\pi)^{-1}\tilde a_S^o$ and 
$$\di\nu_R(v)=(\cosh v-\cosh R)^{-1/2}\sinh v \di v.$$
Thus for all functions $\psi$ in $C_c(\RR)$ we have that
\begin{align}\label{kpsi}
\kD(x)=\du(x)\int_{\RR}\psi(s^2)\,F_{R(x)}(s)\di s \,,
\end{align}
where for all $R>0$ and $s\in \RR$
\begin{equation}\label{FR}
F_{R}(s)=\begin{cases}
a_S^e\,\mathcal D_{1,v}^{m_{\zg}/2}\mathcal D_{2,v}^{m_{\vg}/2}(\nep^{isv})(R)&{\rm{if~}}m_{\zg}{\rm{~is~even}}\\
a_S^o\,\int_{R}^{\infty}\mathcal D_{1,v}^{(m_{\zg}+1)/2}\mathcal D_{2,v}^{m_{\vg}/2}(\nep^{isv})(v)\di\nu_R(v)&{\rm{if~}}m_{\zg}{\rm{~is~odd}}\,.
\end{cases}
\end{equation}
We are interested in finding the asymptotic behaviour of $F_R$ and its first derivative. To do so we need the following technical lemmata.

We denote by $\mathcal S^{\alpha}$ the symbol class
$$\mathcal S^{\alpha}=\{b\in C^{\infty}(\RR):~\|b\|_{\mathcal S^{\alpha},\,k}:=\sup_{s}(1+s^2)^{\frac{-\alpha+k}{2}}\,|b^{(k)}(s)|<\infty,\,{\rm{for~all~}}k\in \NN  \}\,.   $$
\begin{lem}\label{lemma1}
For all integers $p, q$ such that $p+q\geq 1$ and all $v\geq 1$ 
$$\mathcal D^p_{1,v}\,\mathcal D_{2,v}^{q}(\nep^{isv})(v)=\sum_{k=1}^{p+q}s^k\,q_k(v)\,\nep^{-(p+q/2)\,v}\,\nep^{isv}\,,$$
where $q_k$ is in $\mathcal S^0$ for all $k$.
\end{lem}
\begin{proof}
Since $\frac{1}{\sinh v}=\frac{2}{1-\nep^{-2v}}\,\nep^{-v}$, where $\frac{2}{1-\nep^{-2v}}=\sum_{m=0}^{\infty}\nep^{-2mv}$ is in $\mathcal S^0$ for $v>1$, the lemma follows easily by induction on $q$ and $p$ (compare \cite[Lemma 5.3]{MT}).
\end{proof}
\begin{lem}\label{lemma2}
For all integers $p, q$ such that $p+q\geq 1$ and all $v$ in $[0,4]$
$$\mathcal D^p_{1,v}\,\mathcal D_{2,v}^{q}(\nep^{isv})(v)=\sum_{k=1}^{p+q}s^k\,q_k(v)\,v^{k-2(p+q)}\,\nep^{isv}\,,$$
where $q_k$ is in $\mathcal S^0$ for all $k$.
\end{lem}
\begin{proof}
On the interval $[0,4]$ we have that $\frac{1}{\sinh v}=g(v)\,v^{-1}$, where $g$ is in $\mathcal S^0, $ and a similar statement holds for $\frac{1}{\sinh (v/2)}$. The lemma follows easily by induction on $q$ and $p$ (compare \cite[Lemma 5.8]{MT}).
\end{proof}
In the following proposition we study the asymptotic behaviour of $F_R$ and its first derivative.
\begin{prop}\label{F_R}
If $m_{\zg}$ is even, then
$$F_R(s)=\begin{cases}
\nep^{-QR/2}\,\nep^{iRs}\sum_{k=1}^{(n-1)/2}s^k\,q_k(R)&{\rm{if}}~R\geq 1\\
\nep^{-QR/2}\,\nep^{iRs}\sum_{k=1}^{(n-1)/2}s^k\,q_k(R)\,R^{1-n+k}&{\rm{if}}~R< 1\,,
\end{cases}
$$
and
$$\partial_RF_R(s)=\begin{cases}
\nep^{-QR/2}\,\nep^{iRs}\sum_{k=1}^{(n+1)/2}s^k\,p_k(R)&{\rm{if}}~R\geq 1\\
\nep^{-QR/2}\,\nep^{iRs}\sum_{k=1}^{(n+1)/2}s^k\,p_k(R)\,R^{-n+k}&{\rm{if}}~R< 1\,,
\end{cases}
$$
where $q_k$ and $p_k$ are functions in $\mathcal S^0$ uniformly in $R$ for all $k$.

If $m_{\zg}$ is odd, then for every $\tau$ in $[0,1/2]$
$$F_R(s)=\begin{cases}
\nep^{-QR/2}\,\nep^{iRs}\sum_{k=1}^{n/2}s^k\,b_k^{-1/2}(s)&{\rm{if}}~R\geq 1\\
\nep^{-QR/2}\,\nep^{iRs}\sum_{k=1}^{n/2}s^k\,b_k^{\tau-1/2}(s)\,R^{-n+k+\tau+1/2}&{\rm{if}}~R< 1\,,
\end{cases}
$$
and
$$\partial_RF_R(s)=\begin{cases}
\nep^{-QR/2}\,\nep^{iRs}\sum_{k=1}^{n/2+1}s^k\,p_k^{-1/2}(s)&{\rm{if}}~R\geq 1\\
\nep^{-QR/2}\,\nep^{iRs}\sum_{k=1}^{n/2+1}s^k\,p_k^{\tau-1/2}(s)\,R^{-n+k+\tau-1/2}&{\rm{if}}~R< 1\,,
\end{cases}
$$
where  $b_k^{\beta}(s)$ and $p_k^{\beta}(s)$ are functions depending also on $R$ but which are in $\mathcal S^{\beta}$ uniformly in $R$ for all $k$ and $\beta$.
\end{prop}
\begin{proof}
We first study the case when $m_{\zg}$ is even. In this case by (\ref{FR}) we have that
$$F_R(s)=a_S^e\,\mathcal D_{1,v}^{m_{\zg}/2}\mathcal D_{2,v}^{m_{\vg}/2}(\nep^{isv})(R)\,.$$
By Lemma \ref{lemma1} it follows that if $R\geq 1$, then
$$F_R(s)=\nep^{-QR/2}\,\nep^{iRs}\sum_{k=1}^{(n-1)/2}s^k\,q_k(R)\,,$$
while if $R<1$, then
$$F_R(s)=\nep^{-QR/2}\,\nep^{iRs}\sum_{k=1}^{(n-1)/2}s^k\,q_k(R)\,R^{1-n+k}\,,$$
where $q_k$ is in $\mathcal S^0$ for all $k$. 

By deriving with respect to $R$ the expressions above we obtain that if $R\geq 1$, then
\begin{align*}
\partial_RF_R(s)&=\nep^{-QR/2}\,\nep^{iRs}\sum_{k=1}^{(n-1)/2}s^k\big[(-Q/2+is)\,q_k(R)+\partial_Rq_k(R)\big]\\
&= \nep^{-QR/2}\,\nep^{iRs}\sum_{k=1}^{(n+1)/2}s^k\,p_k(R)\,,
\end{align*}
while if $R<1$, then
\begin{align*}
\partial_RF_R(s)&=\,\nep^{-QR/2}\,\nep^{iRs}\sum_{k=1}^{(n-1)/2}s^k\,R^{1-n+k}\,\big[\big(-Q/2+is+(1-n+k)R^{-1}\big)q_k(R)+\partial_Rq_k(R)\big]\\
&=\,\nep^{-QR/2}\,\nep^{iRs}\sum_{k=1}^{(n+1)/2}s^k\,R^{-n+k}\,p_k(R)\,,
\end{align*}
where $p_k$ is in $\mathcal S^0$ for all $k$.

This proves the proposition in the case when $m_{\zg}$ is even.

We now study the case when $m_{\zg}$ is odd. In this case by (\ref{FR}) we deduce that
$$F_R(s)=a_S^o\,\int_{R}^{\infty}\mathcal D_{1,v}^{(m_{\zg}+1)/2}\mathcal D_{2,v}^{m_{\vg}/2}(\nep^{isv})(v)\di\nu_R(v)\,,$$
where $\di\nu_R(v)=(\cosh v-\cosh R)^{-1/2}\sinh v \di v$. By Lemma \ref{lemma1} it follows that if $R\geq 1$, then
\begin{align*}
F_R(s)&=\sum_{k=1}^{n/2}\int_{R}^{\infty} s^k\,q_k(v)\nep^{-(Q/2+1/2)v}\,\nep^{isv}\di\nu_R(v)\\
&=\sum_{k=1}^{n/2}s^k\,\int_{0}^{\infty}q_k(v+R)\nep^{-(Q/2+1/2)(v+R)}\,\nep^{is(v+R)}\times\\
&\phantom{\sum_{k=1}^{n/2}s^k\,\int_{0}^{\infty}}~\times \big(\cosh (v+R)-\cosh R\big)^{-1/2}\sinh(v+R) \di v\\
&=\sum_{k=1}^{n/2}s^k\,\nep^{-QR/2}\,\nep^{isR}\,\int_{0}^{\infty}q_k(v+R)~\frac{\sinh(v+R)}{\nep^{v+R}}\times\\
&\phantom{ \sum_{k=1}^{n/2}s^k\,\nep^{-QR/2}\,\nep^{isR}\,\int_{0}^{\infty} }~\times \big[\big(\cosh (v+R)-\cosh R\big) \nep^{-(v+R)}\big]^{-1/2}\,\nep^{-Qv/2}\,\nep^{isv}\di v\\
&=\sum_{k=1}^{n/2}s^k\,\nep^{-QR/2}\,\nep^{isR}\,b_k^{-1/2}(s)\,,
\end{align*}
by \cite[Lemmata 5.4, 5.5]{MT}, where $b_k^{-1/2}$ is in $\mathcal S^{-1/2}$ uniformly with respect to $R\geq 1$. If $R<1$, then 
\begin{align*}
F_R(s)&=F_R^1(s)+F_R^2(s)\\
&=\int_R^{\infty}\chi(v)\,\mathcal D_{1,v}^{(m_{\zg}+1)/2}\,\mathcal D_{2,v}^{m_{\vg}/2}(\nep^{isv})(v)\di\nu_R(v)\\
&+\int_2^{\infty}\big(1-\chi(v)\big)\,\mathcal D_{1,v}^{(m_{\zg}+1)/2}\,\mathcal D_{2,v}^{m_{\vg}/2}(\nep^{isv})(v)\di\nu_R(v)\,
\end{align*}
where $\chi$ is a $C^{\infty}$ function supported in $[-4,4]$ equal to $1$ in $[-2,2]$. Let $\tau$ be in $[0,1/2]$. By \cite[Lemmata 5.4, 5.5]{MT} $F_R^2$ is in $\mathcal S(\RR)$ uniformly with respect to $R<1$. We now study the behaviour of $F_R^1$, which by Lemma \ref{lemma2} is equal to
\begin{align*}
F_R^1(s)&=\sum_{k=1}^{n/2}s^k\,\int_R^{\infty}\chi(v)\,q_k(v)\,v^{k-n+1}~\frac{\sinh v}{v}\,(\cosh v-\cosh R)^{-1/2}\,\nep^{isv}\di v\,.
\end{align*}
By \cite[Lemma 5.9]{MT} there exists a function $\gamma$ in $\mathcal S^0$ 
such that $\inf_{v\in [0,4]}\gamma(v)>0$ and the previuos sum is equal to
\begin{align*}
&\sum_{k=1}^{n/2}s^k\,\int_R^{\infty}\chi(v)\,q_k(v)\,v^{k-n+1}~\frac{\sinh v}{v}\,\gamma(v)^{-1/2}\,(v+R)^{-1/2}\,(v-R)^{-1/2}\,\nep^{isv}\di v\\
=&\sum_{k=1}^{n/2}s^k\,\int_0^{\infty}\chi(v+R)\,\gamma(v+R)^{-1/2}\,q_k(v+R)\,(v+R)^{k-n+1}~\frac{\sinh (v+R)}{v+R}\times\\
&\phantom{\sum_{k=1}^{n/2}s^k\,\int_0^{\infty}}\,\times (2R+ v)^{-1/2}\,v^{-1/2}\,\nep^{is(v+R)}\di v\\
=&\,\nep^{isR}\,\sum_{k=1}^{n/2}s^k\,\int_0^{\infty}\gamma_{k,\,R}(v)\,(v+R)^{k-n+1}\,(2R+ v)^{-1/2}\,v^{-1/2}\,\nep^{isv}\di v\\
=&\,\nep^{isR}\,\sum_{k=1}^{n/2}s^k\,f_{k,\,R}(s)\,.
\end{align*}
By \cite[Lemma 5.10]{MT} we deduce that $R^{n-k-1/2-\tau}\,f_{k,\,R}$ is in $\mathcal S^{\tau-1/2}$ uniformly with respect to $R<1$. Thus
$$F_R^1(s)= \nep^{-QR/2}\,\nep^{isR}\sum_{k=1}^{n/2}s^k\, R^{-n+k+1/2+\tau}\,b_k^{\tau-1/2}(s)\,, $$
where $b_k^{\tau-1/2}(s),$ depending also on $R,$  is in $\mathcal S^{\tau-1/2}$ uniformly with respect to $R<1$. 

This concludes the study of the asymptotic behaviour of $F_R$ when $m_{\zg}$ is odd. To study the behaviour of its derivative, we derive with respect to $R$ the expressions above. If $R\geq 1$, then 
\begin{align*}
\partial_RF_R(s)&=\sum_{k=1}^{n/2}s^k\,\nep^{-QR/2}\,\nep^{isR}\,(-Q/2+is)\,b_k^{-1/2}(s)\\
&=\sum_{k=1}^{n/2+1}s^k\,\nep^{-QR/2}\,\nep^{isR}\,p_k^{-1/2}(s)\,,
\end{align*}
where $p_k^{-1/2}$ is in $\mathcal S^{-1/2}$ uniformly with respect to $R\geq 1$. If $R<1$ and $\tau\in[0,1/2]$, then
\begin{align*}
\partial_RF_R(s)&=\nep^{-QR/2}\,\nep^{isR}\sum_{k=1}^{n/2}s^k\, R^{-n+k+1/2+\tau}\,b_k^{\tau-1/2}(s)\big(-Q/2+is+(-n+k+1/2+\tau)\,R^{-1}   \big)\\
&=\nep^{-QR/2}\,\nep^{isR}\sum_{k=1}^{n/2+1}s^k\, R^{-n+k-1/2+\tau}\,p_k^{\tau-1/2}(s)\,,
\end{align*}
where $p_k^{\tau-1/2}$ is in $\mathcal S^{\tau-1/2}$ uniformly with respect to $R< 1$. 

This proves the proposition also in the case when $m_{\zg}$ is odd.
\end{proof}

\section{Spectrally localized estimates for the wave propagator}\label{estimatekernel}
In this section we state pointwise estimates for the convolution kernel of spectrally localized wave propagators associated to the Laplacian $L$ on \DR spaces.

Let $t\in \RR$, $\lambda >0,$ and let $\psi$ be an even bump function in $C^{\infty}(\RR)$ supported in $[-2,2]$. If $\lambda \geq 1$ we shall in addition suppose that $\psi$ vanishes on $[-1,1]$. We denote by $k^t_{\lambda}$ the convolution kernel of $m^t_{\lambda}(\D)=\psi\Big(\frac{\sqrt {\D}}{\lambda} \Big)\,\cos (t\,\sqrt{\D})$. We know by (\ref{kpsi}) that
\begin{equation}\label{kernel}
k^t_{\lambda}(x)=\du(x)\,\int_{\RR}\psi \Big(\frac{s}{\lambda} \Big)\,\cos (t\,s)\,F_{R(x)}(s)\di s\,.
\end{equation}
In the following theorem we estimate the kernel $k^t_{\lambda}$.
\begin{teo}\label{Glambda}
The kernel $k^t_{\lambda}$ is of the form
$$k^t_{\lambda}(x)= \du(x)\,\nep^{-QR(x)/2}\,\big[G_{\lambda}\big(R(x),R(x)-t\big)+ G_{\lambda}\big(R(x),R(x)+t\big) \big]\,,$$
where the function $G_{\lambda}$ satisfies for every $N$ in $\NN$ the following estimates:
\begin{itemize}
\item[(i)] if $m_{\zg}$ is even, then
$$\big|G_{\lambda}(R,u)\big|\leq \begin{cases}
C_N\,(1+|\lambda\,u| )^{-N}\,\sum_{k=1}^{(n-1)/2}\lambda^{k+1}&{\rm{if}}~R\geq 1\\
C_N\,(1+|\lambda\,u| )^{-N}\,\sum_{k=1}^{(n-1)/2}R^{1-n+k}\,\lambda^{k+1}&{\rm{if}}~R< 1\,;
\end{cases}
$$
\item[(ii)] if $m_{\zg}$ is odd, then for every $\tau$ in $[0,1/2]$
$$\big|G_{\lambda}(R,u)\big|\leq \begin{cases}
C_N\,(1+|\lambda\,u| )^{-N}\,\lambda^{(n+1)/2}&{\rm{if}}~R\geq 1~{\rm{and}}~\lambda\geq 1\\
C_N\,(1+|\lambda\,u| )^{-N}\,\lambda^{2}&{\rm{if}}~R\geq 1~{\rm{and}}~\lambda< 1\\
C_N\,(1+|\lambda\,u| )^{-N}\,\sum_{k=1}^{n/2}R^{-n+k+1/2+\tau}\,\lambda^{k+1/2+\tau}&{\rm{if}}~R< 1~{\rm{and}}~\lambda\geq 1\\
C_N\,(1+|\lambda\,u| )^{-N}\,R^{-n+3/2+\tau}\,\lambda^{2}&{\rm{if}}~R< 1~{\rm{and}}~\lambda< 1\,,
\end{cases}
$$
\end{itemize}
where the constants $C_N$ depend only on $\tau$ and on the $C^N$-norms of $\psi$.
\end{teo}
\begin{proof}
We consider the case when $m_{\zg}$ is odd. If $R\geq 1$, then by Proposition \ref{F_R} we deduce that
\begin{align*}
k^t_{\lambda}(x)&= \du(x)\,\nep^{-QR(x)/2}\,\sum_{k=1}^{n/2}\int_{\RR}\psi\Big(\frac{s}{\lambda}\Big)\,b^{-1/2}_k(s)\,s^k\,\frac{\nep^{i\big(R(x)-t\big)s}+\nep^{i\big(R(x)+t\big)s}}{2i}  \di s\\
&=\du(x)\,\nep^{-QR(x)/2}\,\big[G_\lambda\big(R(x),\,R(x)-t\big)+G_\lambda\big(R(x),\,R(x)+t\big)\big]\,,
\end{align*}
where $G_\lambda(R,u)=\sum_{k=1}^{n/2}\int_{\RR}\psi\Big(\frac{s}{\lambda}\Big)\,b_k^{-1/2}(s)\,s^k\,\frac{\nep^{i\,u\,s}}{2i}  \di s\,.$ By \cite[Lemma 6.2]{MT} it follows that
\begin{align*}
|G_\lambda(R,\,u)|&\leq \begin{cases}
C_N\,\big(1+|\lambda\,u|\big)^{-N}\,\sum_{k=1}^{n/2}\lambda^{k+1/2}&{\rm{if}} ~\lambda \geq 1\\
C_N\,\big(1+|\lambda\,u|\big)^{-N}\,\sum_{k=1}^{n/2}\lambda^{k+1}&{\rm{if}} ~\lambda < 1
\end{cases}\\
&\leq \begin{cases}
C_N\,\big(1+|\lambda\,u|\big)^{-N}\,\lambda^{(n+1)/2}&{\rm{if}} ~\lambda \geq 1\\
C_N\,\big(1+|\lambda\,u|\big)^{-N}\,\lambda^{2}&{\rm{if}} ~\lambda < 1\,.
\end{cases}
\end{align*}
Take now $\tau$ in $[0,1/2]$. If $R<1$, then by Proposition \ref{F_R} we deduce that
\begin{align*}
k^t_{\lambda}(x)&= \du(x)\,\nep^{-QR(x)/2}\,\sum_{k=1}^{n/2}R(x)^{-n+k+\tau+1/2}\,\int_{\RR}\psi\Big(\frac{s}{\lambda}\Big)\,b^{\tau-1/2}_k(s)\,s^k\,\frac{\nep^{i\big(R(x)-t\big)s}+\nep^{i\big(R(x)+t\big)s}}{2i}  \di s\\
&=\du(x)\,\nep^{-QR(x)/2}\,\big[G_\lambda\big(R(x),\,R(x)-t\big)+G_\lambda\big(R(x),\,R(x)+t\big)\big]\,,
\end{align*}
where $G_\lambda(R,u)=\sum_{k=1}^{n/2}  R^{-n+k+\tau+1/2}\, \int_{\RR}\psi\Big(\frac{s}{\lambda}\Big)\,b^{\tau-1/2}_k(s)\,s^k\,\frac{\nep^{i\,u\,s}}{2i}  \di s\,.$ By \cite[Lemma 6.2]{MT} it follows that
\begin{align*}
|G_\lambda(R,\,u)|&\leq \begin{cases}
C_N\,\big(1+|\lambda\,u|\big)^{-N}\,\sum_{k=1}^{n/2}R^{-n+k+\tau+1/2}\,\lambda^{k+\tau+1/2}&{\rm{if}}~ \lambda \geq 1\\
C_N\,\big(1+|\lambda\,u|\big)^{-N}\,\sum_{k=1}^{n/2}R^{-n+k+\tau+1/2}\,\lambda^{k+1}&{\rm{if}}~ \lambda < 1
\end{cases}\\
&\leq \begin{cases}
C_N\,\big(1+|\lambda\,u|\big)^{-N}\,\sum_{k=1}^{n/2}R^{-n+k+\tau+1/2}\,\lambda^{k+\tau+1/2}&{\rm{if}}~ \lambda \geq 1\\
C_N\,\big(1+|\lambda\,u|\big)^{-N}\,R^{-n+3/2+\tau}\,\lambda^{2}&{\rm{if}}~ \lambda < 1\,.
\end{cases}
\end{align*}
This proves the theorem when $m_{\zg}$ is odd.

The proof in the case when $m_{\zg}$ is even is  similar and easier, which is why we  omit it.
\end{proof}
As a consequence of Theorem \ref{Glambda} we obtain estimates of the $L^1$-norms of the convolution kernel of $\psi\Big(\frac{\sqrt{\D}}{\lambda}\Big)\,\cos\Big(t\frac{\sqrt{\D}}{\lambda}\Big)$. 
\begin{prop}\label{Wtlambda}
Let $W^t_{\lambda}=k^{t/{\lambda}}_{\lambda}$ denote the convolution kernel of $\psi\Big(\frac{\sqrt{\D}}{\lambda}\Big)\,\cos\Big(t\frac{\sqrt{\D}}{\lambda}\Big)$ and let $\varepsilon\ge 0.$ 
\begin{itemize}
\item[(i)] If $\lambda \geq 1$, then
$$\int_S|W^t_{\lambda}(x)|\,R(x)^{\varepsilon}\dir (x)\lesssim \begin{cases} 
\lambda^{-\varepsilon }\,(1+t)^{(n-1)/2+\varepsilon}&{\rm{if~}} t<\lambda\\
\lambda^{-\varepsilon }\,\lambda^{(n-3)/2}\,t^{1+\varepsilon }&{\rm{if~}} t\geq \lambda;
\end{cases}
$$
\item[(ii)] if $\lambda <1$, then
$$\int_S|W^t_{\lambda}(x)|\,R(x)^{\varepsilon}\dir (x)\lesssim  \lambda^{-\varepsilon }\,(1+t)^{1+\varepsilon}\qquad \forall t\in\RR\,.$$
\end{itemize}
In particular, 
\begin{itemize}
\item[(iii)] if $\lambda \geq 1$, then
$$\int_S|W^t_{\lambda}(x)|\,\big(1+\lambda\,R(x) \big)^{\varepsilon}\dir (x)\lesssim  (1+t)^{(n-1)/2+\varepsilon}\qquad \forall t\in\RR\,,$$
\item[(iv)] if $\lambda <1$, then
$$\int_S|W^t_{\lambda}(x)|\,\big(1+\lambda\,R(x) \big)^{\varepsilon}\dir (x)\lesssim  (1+t)^{1+\varepsilon}\qquad \forall t\in\RR\,.$$
\end{itemize}
The constants in these estimates depend only on the $C^N$-norms of $\psi$.
\end{prop}
\begin{proof}
We first assume  that $m_{\zg}$ is odd. Without loss of generality we shall suppose that $t\geq 0$, so that in Theorem \ref{Glambda} the dominant term is $G_{\lambda}(R,R-t)$. Consider first the case when $\lambda\geq 1$. By applying Lemma \ref{intduf} and 
Theorem \ref{Glambda} with $\tau=0$ we deduce that 
\begin{align*}
I&=\int_S|W^t_{\lambda}(x)|\,R(x)^{\varepsilon}\dir (x)\\
&\lesssim \sum_{k=1}^{n/2}\,\int_{R(x)<1}  \du(x)\,\nep^{-QR(x)/2}\,(1+|\lambda\,R(x)-t| )^{-N}\,R(x)^{-n+k+1/2}\,\lambda^{k+1/2}\,R(x)^{\varepsilon}\dir (x)\\
&+\,\int_{R(x)\geq 1}  \du(x)\,\nep^{-QR(x)/2}(1+|\lambda\,R(x)-t| )^{-N}\,\lambda^{(n+1)/2}\,R(x)^{\varepsilon}\dir (x)\\
&\lesssim   \sum_{k=1}^{n/2}\int_0^1(1+|\lambda\,R-t| )^{-N}\,R^{-n+k+1/2}\,\lambda^{k+1/2}\,R^{n-1}\,R^{\varepsilon}\di R\\
&+\int_1^{\infty}(1+|\lambda\,R-t| )^{-N}\,\lambda^{(n+1)/2}\,R^{1+\varepsilon}\di R\\
&\lesssim   \sum_{k=1}^{n/2}\int_0^1(1+|\lambda\,R-t| )^{-N}\,R^{k-1/2+\varepsilon}\,\lambda^{k+1/2}\di R\\
&+\int_1^{\infty}(1+|\lambda\,R-t| )^{-N}\,\lambda^{(n+1)/2}\,R^{1+\varepsilon}\di R\,.
\end{align*}
If $N$ is sufficently large, then (i) follows by \cite[Lemma 6.4]{MT}.

Assume next that  $\lambda <1.$ We then  apply Theorem \ref{Glambda} with $\tau=1/2$ to obtain
\begin{align*}
I&=\int_S|W^t_{\lambda}(x)|\,R(x)^{\varepsilon}\dir (x)\\
&\lesssim \int_{R(x)<1}  \du(x)\,\nep^{-QR(x)/2}\,(1+|\lambda\,R(x)-t| )^{-N}\,R(x)^{-n+3/2+1/2}\,\lambda^{2}\,R(x)^{\varepsilon}\dir (x)\\
&+\,\int_{R(x)\geq 1}  \du(x)\,\nep^{-QR(x)/2}(1+|\lambda\,R(x)-t| )^{-N}\,\lambda^{2}\,R(x)^{\varepsilon}\dir (x)\\
&\lesssim \int_0^1(1+|\lambda\,R-t| )^{-N}\,R^{-n+3/2+1/2}\,\lambda^{2}\,R^{n-1}\,R^{\varepsilon}\di R\\
&+\int_1^{\infty}(1+|\lambda\,R-t| )^{-N}\,\lambda^{2}\,R^{1+\varepsilon}\di R\\
&\lesssim \int_0^{\infty}(1+|\lambda\,R-t| )^{-N}\,\lambda^{2}\,R^{1+\varepsilon}\di R\,.
\end{align*}
If $N$ is sufficiently large, then (ii) follows by \cite[Lemma 6.4]{MT}.

This concludes the proof of (i) and (ii) in the case when $m_{\zg}$ is odd.
The proof in the case when $m_{\zg}$ is even is similar and easier, so that we  omit it.

The estimates (iii) and (iv) are an easy consequence of (i) and (ii).
\end{proof}

By means of the subordination principle described e.g. in \cite{Mu} we immediately obtain the following corollary, which gives a new proof of \cite[Theorem 6.1]{HS} and \cite[Theorem 4.3, estimate (20)]{V1}.
\begin{coro}
Let $\varepsilon, s_0, s_1$ be positive constants such that $s_0>3/2+\varepsilon$ and $s_1>n/2+\varepsilon$. Then there exists a constant $C$ such that for every continuous function $F$ supported in $[1,2]$ and $0<\lambda< 1$,
$$\int_S\Big|k_{F\big(\frac{\D}{\lambda^2}\big)}(x)  \Big|\,\big(1+\lambda\,R(x) \big)^{\varepsilon}\,\dir (x)\leq C\,\|F\|_{H^{s_0}(\RR)}\,,$$
while for $\lambda\geq 1$
$$\int_S\Big|k_{F\big(\frac{\D}{\lambda^2}\big)}(x)  \Big|\,\big(1+\lambda\,R(x) \big)^{\varepsilon}\,\dir (x)\leq C\,\|F\|_{H^{s_1}(\RR)}\,.$$
\end{coro}
\begin{proof}
The proof follows along the lines  of the proof of \cite[Corollary 6.5]{MT}. We omit the details.
\end{proof}

\section{Spectrally localized estimates for the gradient of the wave propagator}In this section we prove an estimate of the gradient of the wave propagator associated to the Laplacian $L$.

Let $k^t_{\lambda}$ denote the convolution kernel of $m^t_{\lambda}(\D)=\psi\Big(\frac{\sqrt {\D}}{\lambda} \Big)\,\cos (t\,\sqrt{\D})$ and $X_0,\ldots,X_{n-1}$ be the left invariant vector fields which we introduced in Section \ref{Laplacian}. We have that
\begin{align*}
X_0k^t_{\lambda}(x)&=\big(X_0\du\big)(x)\,\int_{\RR}\psi \Big(\frac{s}{\lambda} \Big)\,\cos (t\,s)\,F_{R(x)}(s)\di s\\
&+\du(x)\,\int_{\RR}\psi \Big(\frac{s}{\lambda} \Big)\,\cos (t\,s)\,\partial_RF_{R(x)}(s)\,(X_0R)(x)\di s
\\
=&\du(x)\,\int_{\RR}\psi \Big(\frac{s}{\lambda} \Big)\,\cos (t\,s)\,\big[-Q/2\,F_{R(x)}(s)+    \partial_RF_{R(x)}(s)\,(X_0R)(x)\big]\di s\,,
\end{align*}
and
$$X_ik^t_{\lambda}(x)=\du(x)\,\int_{\RR}\psi \Big(\frac{s}{\lambda} \Big)\,\cos (t\,s)\,\partial_RF_{R(x)}(s)\,(X_iR)(x)\di s\qquad \forall i=1,\ldots,n-1\,.
$$
We recall that $\big|X_iR(x)\big|\leq 1$, for all $x\in S$ and $i=0,\ldots,n-1$ \cite{H}.

\begin{teo}\label{Hlambda}
For all $i=0,..,n-1$ we have that
$$X_ik^t_{\lambda}(x)= \du(x)\,\nep^{-QR(x)/2}\,\big[H_{\lambda}\big(R(x),R(x)-t\big)+ H_{\lambda}\big(R(x),R(x)+t\big) \big]\,,$$
where the function $H_{\lambda}$ satisfies for every $N$ in $\NN$ the following estimates:
\begin{itemize}
\item[(i)] if $m_{\zg}$ is even, then
$$\big|H_{\lambda}(R,u)\big|\leq  \begin{cases}
C_N\,(1+|\lambda\,u| )^{-N}\,\sum_{k=1}^{(n+1)/2}\lambda^{k+1}&{\rm{if}}~R\geq 1\\
C_N\,(1+|\lambda\,u| )^{-N}\,\sum_{k=1}^{(n+1)/2}R^{-n+k}\,\lambda^{k+1}&{\rm{if}}~R< 1\,;
\end{cases}
$$
\item[(ii)] if $m_{\zg}$ is odd, then for every $\tau$ in $[0,1/2]$
$$\big|H_{\lambda}(R,u)\big|\leq  \begin{cases}
C_N\,(1+|\lambda\,u| )^{-N}\,\lambda^{(n+3)/2}&{\rm{if}}~R\geq 1~{\rm{and}}~\lambda\geq 1\\
C_N\,(1+|\lambda\,u| )^{-N}\,\lambda^{2}&{\rm{if}}~R\geq 1~{\rm{and}}~\lambda< 1\\
C_N\,(1+|\lambda\,u| )^{-N}\,\sum_{k=1}^{n/2+1}R^{-n+k-1/2+\tau}\,\lambda^{k+1/2+\tau}&{\rm{if}}~R< 1~{\rm{and}}~\lambda\geq 1\\
C_N\,(1+|\lambda\,u| )^{-N}\,R^{-n+1/2+\tau}\,\lambda^{2}&{\rm{if}}~R< 1~{\rm{and}}~\lambda< 1\,,
\end{cases}
$$
\end{itemize}
where the constants $C_N$ depend only on $\tau$ and on the $C^N$-norms of $\psi$.
\end{teo}
\begin{proof}
As before, we shall only consider the case when  $m_{\zg}$ is odd;   the case when $m_{\zg}$ is even is similar and easier. 
Moreover, we shall only discuss the case where $i=0$. When $i=1,\ldots,n-1$ the proof is again  similar and easier, which is why  we omit it. 

If $R\geq 1,$ then  by Proposition \ref{F_R} we deduce that
\begin{align*}
X_0k^t_{\lambda}(x)&= \du(x)\,\nep^{-QR(x)/2}\,\sum_{k=1}^{n/2+1}\int_{\RR}\psi\Big(\frac{s}{\lambda}\Big)\,\Big[-\frac{Q}{2}\,b^{-1/2}_k(s)+p_k^{-1/2}(s)\,(X_0R)(x)\Big]\times\\
&\phantom{\du(x)\,\nep^{-QR(x)/2}\,\sum_{k=1}^{n/2+1}\int_{\RR}} \times \,s^k\,\frac{\nep^{i\big(R(x)-t\big)s}+\nep^{i\big(R(x)+t\big)s}}{2i}  \di s\\
&=\du(x)\,\nep^{-QR(x)/2}\,\big[H_\lambda\big(R(x),\,R(x)-t\big)+H_\lambda\big(R(x),\,R(x)+t\big)\big]\,,
\end{align*}
where $H_\lambda(R,u)=\sum_{k=1}^{n/2+1}\int_{\RR}\psi\Big(\frac{s}{\lambda}\Big)\,  \Big[-\frac{Q}{2}\,b^{-1/2}_k(s)+p_k^{-1/2}(s)\,(X_0R)\big]\,s^k\,\frac{\nep^{i\,u\,s}}{2i}  \di s\,.$ By \cite[Lemma 6.2]{MT} it follows that
\begin{align*}
|H_\lambda(R,\,u)|&\leq \begin{cases}
C_N\,\big(1+|\lambda\,u|\big)^{-N}\,\sum_{k=1}^{n/2+1}\lambda^{k+1/2}&{\rm{if}}~ \lambda \geq 1\\
C_N\,\big(1+|\lambda\,u|\big)^{-N}\,\sum_{k=1}^{n/2+1}\lambda^{k+1}&{\rm{if}}~ \lambda < 1,
\end{cases}
\end{align*}
which gives the required estimates in (ii) when $R\ge 1.$ 
 If $R<1$, and if $\tau\in [0,1/2],$ then
\begin{align*}
X_0k^t_{\lambda}(x)&= \du(x)\,\nep^{-QR(x)/2}\,\sum_{k=1}^{n/2+1}\int_{\RR}\psi\Big(\frac{s}{\lambda}\Big)\,\Big[-\frac{Q}{2}\,R(x)^{-n+k+\tau+1/2}\,b^{\tau-1/2}_k(s)\\
&\phantom{\du(x)\,\nep,}+ R(x)^{-n+k+\tau-1/2}\,p^{\tau-1/2}_k(s)(X_0R)(x)\Big]  \,s^k\,\frac{\nep^{i\big(R(x)-t\big)s}+\nep^{i\big(R(x)+t\big)s}}{2i}  \di s\\
&=\du(x)\,\nep^{-QR(x)/2}\,\big[H_\lambda\big(R(x),\,R(x)-t\big)+H_\lambda\big(R(x),\,R(x)+t\big)\big]\,,
\end{align*}
where 
$$H_\lambda(R,u)=\sum_{k=1}^{n/2+1} \int_{\RR}\psi\Big(\frac{s}{\lambda}\Big)\,R^{-n+k+\tau-1/2}    \big[-Q/2\,R\,b^{\tau-1/2}_k(s)+ p^{\tau-1/2}_k(s)(X_0R)\big]                   \,s^k\,\frac{\nep^{i\,u\,s}}{2i}  \di s\,.$$ By \cite[Lemma 6.2]{MT} it follows that
\begin{align*}
|H_\lambda(R,\,u)|&\leq \begin{cases}
C_N\,\big(1+|\lambda\,u|\big)^{-N}\,\sum_{k=1}^{n/2+1}R^{-n+k+\tau-1/2}\,\lambda^{k+\tau+1/2}&{\rm{if}}~ \lambda \geq 1\\
C_N\,\big(1+|\lambda\,u|\big)^{-N}\,\sum_{k=1}^{n/2+1}R^{-n+k+\tau-1/2}\,\lambda^{k+1}&{\rm{if}}~ \lambda < 1,
\end{cases}
\end{align*}
which yields the required estimates in (ii) also  when $R<1.$

\end{proof}

As a consequence of Theorem \ref{Hlambda} we obtain estimates of the $L^1$-norms of the gradient of the convolution kernel of $\psi\Big(\frac{\sqrt{\D}}{\lambda}\Big)\,\cos\Big(t\frac{\sqrt{\D}}{\lambda}\Big)$. To do it we need the following technical lemma.
\begin{lem}\label{integral}
The following estimates hold:
\begin{itemize}
\item[(i)] for all $x=(v,z,a)$ in $S$ 
$$\Big|X_0\Big(\du\,\big(\cosh(R/2)  \big)^{-Q}  \Big) \big(v,z,a\big)| \leq \frac{a\,(a+1+|v|^2/4)^{-Q-1}}{\big[1+ (a+1+|v|^2/4)^{-2} \,|z|^2\big]^{Q/2+1}}\,;$$ 
\item[(ii)] $ \int_N \frac{(a+1+|v|^2/4)^{-Q-1}}{\big[1+ (a+1+|v|^2/4)^{-2} \,|z|^2\big]^{Q/2+1}}
\di v \di z= C\,(a+1)^{-1}  \qquad \forall a\in\RR^+ \,.$
\end{itemize}
\end{lem}
\begin{proof}
For a proof of (i) see \cite[Lemma 4.5]{V1}. In order to prove (ii), notice that  the  change variables $(a+1+|v|^2/4)^{-1}\,z=w$ in  the integral on the left-hand side of (ii) transforms it into 
\begin{align*}
&\int_{\vg} (a+1+|v|^2/4)^{-Q-1+m_{\zg}}\di v\,\int_{\zg}\frac{\di w}{(1+|w|^2)^{Q/2+1}}\\
\lesssim& \int_{\vg} (a+1+|v|^2/4)^{-m_{\vg}/2-1}\di v \\
\lesssim&\,C\,  (a+1)^{-1}   \,,
\end{align*}
as required.
\end{proof}
We can now  prove an $L^1$-estimate for the Riemannian gradient of the convolution kernel of $\psi\Big(\frac{\sqrt{\D}}{\lambda}\Big)\,\cos\Big(t\frac{\sqrt{\D}}{\lambda}\Big)$.
\begin{prop}\label{gradient}
Let $W^t_{\lambda}=k^{t/{\lambda}}_{\lambda}$ denote the convolution kernel of $\psi\Big(\frac{\sqrt{\D}}{\lambda}\Big)\,\cos\Big(t\frac{\sqrt{\D}}{\lambda}\Big)$. 
\begin{itemize}
\item[(i)] If $\lambda \geq 1$ and $\supp \psi\subset [-2,-1]\cup [1,2]$ then 
$$\int_S\|\nabla W^t_{\lambda}(x)\|\dir (x)\leq\begin{cases}  
C\,\lambda\,(1+t)^{(n-1)/2}&{\rm{if~}}t<\lambda\\
C\,\lambda^{(n-1)/2}\,t&{\rm{if~}}t\geq \lambda\,.
\end{cases}
$$
\item[(ii)] If $\lambda <1$ and $\supp \psi\subset [-2,2]$ then
$$\int_S\|\nabla W^t_{\lambda}(x)\|\dir (x)\leq C\, \lambda\,(1+t)\qquad \forall t\in\RR\,.$$
\end{itemize}
The constants in these estimates depend only on the $C^N$-norms of $\psi$.
\end{prop}
\begin{proof}
Again, we shall only deal with the case when  $m_{\zg}$ is odd, since the other case is similar and easier. Without loss of generality we may and shall assume that $t\geq 0$ so that the dominant term in Theorem \ref{Hlambda} is $H_{\lambda}(R,R-t)$. We first observe that
\begin{align*}
I&=\int_S\|\nabla W^t_{\lambda}(x)\|\dir (x)\\
&=\int_{R(x)<1}\|\nabla W^t_{\lambda}(x)\|\dir (x)+\int_{R(x)\geq 1}\|\nabla W^t_{\lambda}(x)\|\dir (x)\\
&=I_0+I_{\infty}\,,
\end{align*}
and we estimate these integrals separately. 

{\emph{1. Case: $\lambda\geq 1$.}} By applying Lemma \ref{intduf} and Theorem \ref{Hlambda} with $\tau=0$ we obtain that
\begin{align*}
I_0&\lesssim \int_{R(x)<1}  \du(x)\,\nep^{-QR(x)/2}\,\big|H_{\lambda}\big(R(x),R(x)-t/{\lambda}\big)\big|\dir (x)\\
&\lesssim \int_0^1R^{n-1}\big|H_{\lambda}\big(R,R-t/{\lambda}\big)\big|\di R\\
&\lesssim  \sum_{k=1}^{n/2+1}\int_0^1R^{n-1}\,(1+|\lambda\,R-t| )^{-N}\,R^{-n+k-1/2}\,\lambda^{k+1/2}\di R\,,
\end{align*}
which by \cite[Lemma 6.4]{MT} satisfies (i).

We proceed in the same way for $I_{\infty}$.   
By Theorem \ref{Hlambda} and Lemma \ref{intduf} 
\begin{align*}
I_{\infty}&\lesssim \int_{R(x)\geq 1}  \du(x)\,\nep^{-QR(x)/2}\,\big|H_{\lambda}\big(R(x),R(x)-t/{\lambda}\big)\big|\dir (x)\\
&\lesssim \int_1^{\infty}R\,\big|H_{\lambda}\big(R,R-t/{\lambda}\big)\big|\di R\\
&\lesssim \int_1^{\infty}R\,(1+|\lambda\,R-t| )^{-N}\,\lambda^{(n+3)/2}\di R\,.
\end{align*}
By \cite[Lemma 6.4]{MT} the estimate (i) follows.


{\emph{2. Case: $\lambda< 1$.}} In this case the estimate of $I_0$ is easy. Indeed, by applying Lemma \ref{intduf} and Theorem \ref{Hlambda} with $\tau=1/2$ we deduce that 
\begin{align*}
I_0&\lesssim \int_0^1R^{n-1}\,R^{-n+1/2+1/2}\,\lambda^{2}\,\big(1+|\lambda\,R-t|\big)^{-N}\di R\\
&\lesssim \int_0^1\lambda^{2}\,\big(1+|\lambda\,R-t|\big)^{-N}\di R\,,
\end{align*}
which is clearly bounded above by $\lambda\,(1+t).$  

The estimation of the integral $I_{\infty}$ is more delicate. To obtain it we need a more precise study of the behaviour of $\nabla W^t_{\lambda}$ when $R\geq 1$ and $\lambda<1$. We recall that by (\ref{kpsi}) and Proposition \ref{F_R}
\begin{align*}
W^t_{\lambda}(x)&=\du(x)\,\nep^{-QR(x)/2}\,\sum_{k=1}^{n/2}\int_{\RR}\psi\Big(\frac{s}{\lambda}  \Big)\,\cos\Big(\frac{t\,s}{\lambda}\Big)\,s^k\,b_k^{-1/2}(s)\,\nep^{i\,R(x)\,s}\di s\\
&=\sum_{k=1}^{n/2}L_k(x)\,,
\end{align*}
where $L_k(x)=\du(x)\,\nep^{-QR(x)/2}\,\int_{\RR}\psi\Big(\frac{s}{\lambda}  \Big)\,\cos\Big(\frac{t\,s}{\lambda}\Big)\,s^k\,b_k^{-1/2}(s)\,\nep^{i\,R(x)\,s}\di s\,.$

Since $I_{\infty}\lesssim \sum_{k=1}^{n/2}\int_{R(x)\geq 1}\|\nabla L_k(x)\|\dir(x)$ we estimate each summand separately. More precisely we distinguish the cases when $k\geq 2$ and $k=1$ (which corresponds to the worst summand).

Let us suppose that $k\geq 2$ and  consider the derivative of $L_k$ along the vector field $X_0$:
\begin{align*}
X_0L_k(x)&=\big(X_0\du\big)(x)\,\nep^{-QR(x)/2}\,\int_{\RR}\psi\Big(\frac{s}{\lambda}  \Big)\,\cos\Big(\frac{t\,s}{\lambda}\Big)\,s^k\,b_k^{-1/2}(s)\,\nep^{i\,R(x)\,s}\di s\\
&-Q/2\,\big(X_0R\big)(x)\,\du(x)\,\nep^{-QR(x)/2}\times\\
&\times \,\int_{\RR}\psi\Big(\frac{s}{\lambda}  \Big)\,\cos\Big(\frac{t\,s}{\lambda}\Big)\,s^k\,b_k^{-1/2}(s)\,\nep^{i\,R(x)\,s}\di s\\
&+\big(X_0R\big)(x)\,\du(x)\nep^{-QR(x)/2}\,\int_{\RR}\psi\Big(\frac{s}{\lambda}  \Big)\,\cos\Big(\frac{t\,s}{\lambda}\Big)\,is^{k+1}\,b_k^{-1/2}(s)\,\nep^{i\,R(x)\,s}\di s\,.
\end{align*}
Thus

\begin{align}\label{Lkodd}
\big|X_0L_k(x)\big|&\lesssim  \du(x)\,\nep^{-QR(x)/2}\Big[\big|L_{\lambda}^k\big(R(x),R(x)-t/{\lambda}  \big)\big|+ \big|L_{\lambda}^k\big(R(x),R(x)+t/{\lambda}  \big)\big| \Big]\nonumber\\
&+\du(x)\,\nep^{-QR(x)/2}\Big[\big|L_{\lambda}^{k+1}\big(R(x),R(x)-t/{\lambda}  \big)\big|+ \big|L_{\lambda}^{k+1}\big(R(x),R(x)+t/{\lambda}  \big)\big| \Big]\,,
\end{align}
where $L_{\lambda}^k(R,u)=\int_{\RR}\psi\Big(\frac{s}{\lambda}\Big)\,s^k\,b_k^{-1/2}(s)~\frac{\nep^{i\,u\,s}}{2i}   \di s$. 
Since $k\geq 2$ by \cite[Lemma 6.2]{MT} it follows that for all $N$ in $\NN$
\begin{equation}\label{Lk'odd}
\big|L_{\lambda}^k(R,u) \big|+\big|L_{\lambda}^{k+1}(R,u) \big|\leq C_N\,(1+|\lambda\,u| )^{-N}\,(\lambda^{k+1}+\lambda^{k+2})\,\leq C_N\,(1+|\lambda\,u| )^{-N}\,\lambda^{3}\,.
\end{equation}
Thus by (\ref{Lkodd}) and (\ref{Lk'odd})
\begin{align}\label{1}
\int_{R(x)\geq 1}\big|X_0L_k(x)\big|\dir (x)&\lesssim\int_1^{\infty} (1+|\lambda R-t| )^{-N}\,\lambda^{3}\,R\di R\nonumber\\
&\leq \lambda\,(1+t)\,.
\end{align}
For $i=1,...n-1$ the conclusion follows in the same way so that
$$
\sum_{k=2}^{n/2}\int_{R(x)\geq 1}\|\nabla L_k(x)\|\dir(x)\lesssim \lambda\,(1+t)\,. 
$$
It remains to consider the case when $k=1$. We first write $L_1$ in the following form:
\begin{align*}
L_1(x)&=\du(x)\big(\cosh(R(x)/2)  \big)^{-Q}\,\Big(\frac{\cosh(R(x)/2)}{\nep^{R(x)/2}}\Big)^{Q}\,\int_{\RR}\psi\Big(\frac{s}{\lambda}  \Big)\,\cos\Big(\frac{t\,s}{\lambda}\Big)\,s\,b_1^{-1/2}(s)\,\nep^{i\,R(x)\,s}\di s\\
&=\du(x)\big(\cosh(R(x)/2)  \big)^{-Q}\,g\big(R(x)\big)\,h\big(R(x)\big)\,,
\end{align*}
where the function $g$ is in $\mathcal S^0$ for   $R\geq 1,$ and where 
$$
h(R)=\int_{\RR}\psi\Big(\frac{s}{\lambda}  \Big)\,\cos\Big(\frac{t\,s}{\lambda}\Big)\,s\,b_1^{-1/2}(s)\,\nep^{i\,R\,s}\di s\,.
$$ 
By deriving along the vector field $X_0$ we obtain that
\begin{align*}
X_0L_1(x)&=X_0\Big(\du(x)\big(\cosh(R(x)/2)  \big)^{-Q}  \Big)\,g\big(R(x)\big)\,h\big(R(x)\big)\\
&+\du(x)\big(\cosh(R(x)/2)  \big)^{-Q}\,\big[\partial_Rg\big(R(x)\big)\,h\big(R(x)\big)\,+g\big(R(x)\big)\,\partial_Rh\big(R(x)\big)\big]\,\big(X_0R\big)(x)\\
&=A(x)+B(x)\,.
\end{align*}
An easy application of \cite[Lemma 6.2]{MT} shows that for all $N$ in $\NN$
\begin{align}\label{Bodd}
|B(x)|&\leq C_N\,\du(x)\,\nep^{-QR(x)/2}\,\big[R(x)^{-1}\,\lambda^2+\lambda^3\big ]\,(1+|\lambda \,R(x)-t| )^{-N}\,,
\end{align}
and
\begin{align}\label{Aodd}
|A(x)|&\leq C_N\,   \Big|X_0\Big(\du(x)\big(\cosh(R(x)/2)  \big)^{-Q}  \Big) \Big|  \,\lambda^2\,(1+|\lambda \,R(x)-t| )^{-N}\,.
\end{align}
It suffices to prove that the integrals of $|A|$ and $|B|$ over the complement of the unit ball satisfy the required estimate.

By (\ref{Bodd}) and \cite[Lemma 6.4]{MT} we deduce that  
\begin{align}\label{2}
\int_{R(x)\geq 1}|B(x)|\dir (x)&\lesssim   \int_1^{\infty}R\,\big[R^{-1}\,\lambda^2+\lambda^3]\,(1+|\lambda \,R-t| )^{-N}\di R\nonumber\\
&\lesssim \begin{cases}
\lambda^2\,\lambda^{-1}\,(1+\lambda)^{-N+1}+\lambda^3\,\lambda^{-2}\,(1+\lambda)^{-N+1}&{\rm{if}}~t<\lambda/2\\
\lambda^2\,\lambda^{-1}+\lambda^3\,\lambda^{-2}\,(1+t)&{\rm{if}}~t\geq \lambda/2
\end{cases}\nonumber\\
&\lesssim \lambda \,(1+t)\,,
\end{align}
as required. We next integrate $|A|$:
\begin{align*}
\int_{R(x)\geq 1}\big|A(x)\big|\dir(x)&\leq \int_{R(x)< 2t/{\lambda}}\big|A(x)\big|\dir(x)+ \int_{R(x)\geq  2t/{\lambda}}\big|A(x)\big|\dir(x)\\
&= J_1+J_2\,.
\end{align*}
We estimate $J_1$ and $J_2$ separately by using Lemma \ref{integral} above. Note that if $R(x)\geq  2t/{\lambda}$, then
\begin{align*}
(1+|\lambda \,R(x)-t| )^{-N}&\lesssim (1+\lambda \,R(x) )^{-N}
\ \lesssim\  \big(1+\lambda \,\log {\rm{max}}(a,\,1/a) \big)^{-N}\,,
\end{align*}
since by \eqref{distanza} 
$$ \log {\rm{max}}(a,\,1/a)\lesssim R(x),$$
 if $a$ denotes the $A$-component of $x\in S.$
Thus, by Lemma \ref{integral},
\begin{align*}
J_2&\lesssim \lambda^2\, \int_{\RR^+} \frac{\di a}{a} \, \Big(1+\lambda \,\log {\rm{max}}\big(a,\,1/a\big) \Big)^{-N}
\,a\,\int_N \frac{(a+1+|v|^2/4)^{-Q-1}}{\big[1+ (a+1+|v|^2/4)^{-2} \,|z|^2\big]^{Q/2+1}}
\di v \di z\\
&\lesssim \lambda^2\, \int_{\RR^+} {\di a} \, \Big(1+\lambda \,\log {\rm{max}}\big(a,\,1/a\big) \Big)^{-N}
\,  (a+1)^{-1} \\
&\lesssim\int_1^\infty (1+\lambda \log a)^{-N} a^{-1}\, da
\ \lesssim \  \lambda\,(t+1)\,.
\end{align*}
On the other hand, we note that if $R(x) <2t/{\lambda}$, then similarly
\begin{align*}
|\log a |\lesssim R(x)<2t/{\lambda}\,,
\end{align*}
so that by Lemma \ref{integral}
\begin{align*}
J_1&\lesssim  \lambda^2\, \int_{|\log a |\lesssim 2t/{\lambda}}\frac{\di a}{a}\,a\,\int_N \frac{(a+1+|v|^2/4)^{-Q-1}}{\big[1+ (a+1+|v|^2/4)^{-2} \,|z|^2\big]^{Q/2+1}}\di v \di z\\
&\lesssim  \lambda^2\, \int_{|\log a |\lesssim 2t/{\lambda}}a^{-1}\di a
\ \lesssim \  \lambda^2\,2t/{\lambda}
\ \lesssim \ \lambda\,(1+t)\,.
\end{align*}
It follows that 
\begin{equation}\label{3}
\int_{R(x)\geq 1}|A(x)|\dir(x)\lesssim \lambda\,(1+t)\,.
\end{equation} 
By (\ref{1}), (\ref{2}), (\ref{3}) we deduce that $I_{\infty}\lesssim   \lambda\,(1+t)\,$.

This concludes the proof also in the case when $\lambda<1$.

\end{proof}

By means of the subordination principle described e.g. in \cite{Mu} we obtain the following corollary which gives a new proof of \cite[Theorem 4.3]{V1}.

\begin{coro}
Let $s_0, s_1$ be positive constants such that $s_0>3/2$ and $s_1>n/2$. Then there exists a constant $C$ such that for every continuous function $F$ supported in $[1,2]$ 
$$\int_S\Big|K_{F\big(\frac{\D}{\lambda^2}\big)}(x,y) -K_{F\big(\frac{\D}{\lambda^2}\big)}(x,z) \Big|\dir (x)\leq \begin{cases}
C\,\lambda\,d(y,z)\,\|F\|_{H^{s_0}(\RR)}&{\rm{if}} ~\lambda< 1\\
C\,\lambda\,d(y,z)\,\|F\|_{H^{s_1}(\RR)}&{\rm{if}} ~\lambda\geq 1\,.
\end{cases}
$$
\end{coro}
\begin{proof}
By (\ref{integralkernel}) the integral kernel $K_{F\big(\frac{\D}{\lambda^2}\big)}$ is given by
$$K_{F\big(\frac{\D}{\lambda^2}\big)}(x,y)=\delta(y)\,k_{F\big(\frac{\D}{\lambda^2}\big)}(y^{-1}\,x)\,.$$
Since $\D$ is symmetric it suffices to prove the analogous estimates for $\int_S\Big|K_{F\big(\frac{\D}{\lambda^2}\big)}(y,x) -K_{F\big(\frac{\D}{\lambda^2}\big)}(z,x) \Big|\dir (x),$ since $K_{F\big(\frac{\D}{\lambda^2}\big)}(x,y)=\overline{K_{\overline F\big(\frac{\D}{\lambda^2}\big)}(y,x)}.$ But 
\begin{align*}
&\int_S\Big|K_{F\big(\frac{\D}{\lambda^2}\big)}(y,x) -K_{F\big(\frac{\D}{\lambda^2}\big)}(z,x) \Big|\dir (x)\\
&\phantom{}=\int_S\Big|\delta(x)\,k_{F\big(\frac{\D}{\lambda^2}\big)}(x^{-1}y) -\delta(x)\,k_{F\big(\frac{\D}{\lambda^2}\big)}(x^{-1}z) \Big|\dir (x)\\
&= \int_S\Big|k_{F\big(\frac{\D}{\lambda^2}\big)}(xy) -k_{F\big(\frac{\D}{\lambda^2}\big)}(xz) \Big|\dir (x)\\
&= \int_S\Big|k_{F\big(\frac{\D}{\lambda^2}\big)}(xz^{-1}y) -k_{F\big(\frac{\D}{\lambda^2}\big)}(x) \Big|\dir (x)\\
&\leq d(z^{-1}y,e)\,\int_S\|\nabla k_{F\big(\frac{\D}{\lambda^2}\big)}(u)\|\dir (u)\\
&= d(y,z)\,\|\nabla k_{F\big(\frac{\D}{\lambda^2}\big)}\|_{L^1(\rho)}\,.
\end{align*}
Now choose an even function $\psi$ in $C^{\infty}_0(\RR)$ such that $\supp  \psi\subset [-4,-1/2]\cup [1/2,4]$ and $\psi=1$ on $[1,2]$. Set $f(v)=F(v^2)$. Then $\|f\|_{H^s}\sim\|F\|_{H^s}$ and $ F\Big(\frac{\D}{\lambda^2}\Big)=f\Big(\frac{\sqrt{\D}}{\lambda}\Big)=\psi\Big(\frac{\sqrt{\D}}{\lambda}\Big)\,f\Big(\frac{\sqrt{\D}}{\lambda}\Big)$. Moreover, by the Fourier inverse formula and Fubini's theorem, one easily obtains that 
$$f\Big(\frac{\sqrt{\D}}{\lambda}\Big)=\frac{1}{\pi}\int_{0}^{\infty}\hat{f}(t)\,\cos\Big(t\,\frac{\sqrt{\D}}{\lambda}\Big)\di t\,,$$
since $f$ is an even function. Thus
$$F\Big(\frac{\D}{\lambda^2}\Big)=\frac{1}{\pi}\int_{0}^{\infty}\hat{f}(t)\,\psi\Big(\frac{\sqrt{\D}}{\lambda}\Big)\,\cos\Big(t\,\frac{\sqrt{\D}}{\lambda}\Big)\di t\,,$$ 
which implies
\begin{align*}
\|\nabla k_{F\big(\frac{\D}{\lambda^2}\big)}\|_{L^1(\rho)}&\lesssim \int_{0}^{\infty}|\hat{f}(t)|\,\int_S\|\nabla W^t_{\lambda}(x)\|\dir (x) \di t\,. 
\end{align*}
If $0<\lambda<1$, then by Proposition \ref{gradient}
\begin{align*}
\|\nabla k_{F\big(\frac{\D}{\lambda^2}\big)}\|_{L^1(\rho)}&\lesssim\int_{0}^{\infty}|\hat{f}(t)|\,\lambda\,(1+|t|)\di t\\
&\leq\lambda\,\Big( \int_{0}^{\infty}\big|\hat{f}(t)\,(1+|t|)^{s_0}\big|^2\di t \Big)^{1/2}\,\Big( \int_{0}^{\infty} (1+|t|)^{2-2s_0}\di t\Big)^{-1/2}\\
&\leq \lambda\, \|f\|_{H^{s_0}}
\lesssim \lambda\, \|F\|_{H^{s_0}}\,,
\end{align*}
if $s_0>3/2$. This proves the corollary when $0<\lambda<1$. The proof in the case when $\lambda\geq 1$ is similar and omitted.
\end{proof}

\begin{remark}
The estimates in Theorem \ref{Glambda} are good enough for $L^1$--estimates, but not for $L^{\infty}$--estimates, since they exhibit singularities at $R=0$. One knows that the singular support of the wave propagator for time $t$ is the sphere $R=t$, so that these singularities are in fact not present. As in \cite[Section 7]{MT} we shall improve the estimates of $k_{\lambda}^{t}$ when $R<1$. 

More precisely we can prove that for $R<1$ the function $G_{\lambda}$ which appears in Theorem \ref{Glambda} satisfies the following estimates for any $N$ in $\NN$:
\begin{align}\label{newGlambda}
\big|G_{\lambda}(R,u)\big|\leq \begin{cases}
C_N\,(1+|\lambda\,u| )^{-N}\,\lambda^{n}&{\rm{if}}~\lambda \geq 1\\
C_N\,(1+|\lambda\,u| )^{-N}\,\lambda^{2}&{\rm{if}}~\lambda< 1\,.
\end{cases}
\end{align}
The proof of (\ref{newGlambda}) follows the same outline as the proof of \cite[Theorem 7.1]{MT}.

By (\ref{newGlambda}) and Theorem \ref{Glambda} it follows easily that for $\lambda\geq 1$ and $t\geq 0$ 
$$\|k^t_{\lambda}\|_{\infty} \lesssim\, (1+t^{-(n-1)/2})\,\lambda^{(n+1)/2}\,.$$
Notice that, for small times, this estimate agrees with the one valid for the Laplacian on the Euclidean space $\RR^n$. However, for large times, there appears no dispersive effect. 
\end{remark}

\section{Growth estimates for solutions to the wave equation in terms of spectral Sobolev norms}
Given a symbol $m$ in $\mathcal S^{-\alpha}$ with $\alpha\ge 0,$  consider  the operators $T^t_1:=m(\sqrt{\D})\,\cos(t\,\sqrt{\D})$ and $T^t_2:=m(\sqrt{\D})\,\frac{\sin(t\,\sqrt{\D})}{\sqrt{\D}}$, for $t\in\RR$, which are bounded on $L^2(\rho)$ by the spectral theorem. We look for a condition on $\alpha$ such that these operators are bounded on $L^p(\rho)$, for some $1\leq p\leq \infty$. Throughout this section we often write $L^p$ instead of $L^p(\rho)$.
\begin{teo}\label{Tt1Tt2}
Let $m$ be a symbol in $\mathcal S^{-\alpha}$ and $1\leq p\leq \infty$. Set $\alpha_n(p):=(n-1)\,\big|1/p-1/2\big|$.
\begin{itemize}
\item[(i)] If $\alpha>\alpha_n(p)$, then $T^t_1$ extends from $L^p(\rho)\cap L^2(\rho)$ to a bounded operator on $L^p(\rho)$, and
$$\|T^t_1\|_{L^p\to L^p}\leq C_p\,(1+|t| )^{2|\frac 1p-\frac 12|}\,.
$$
\item[(ii)] If $\alpha>\alpha_n(p)-1$, then $T^t_2$ extends from $L^p(\rho)\cap L^2(\rho)$ to a bounded operator on $L^p(\rho)$, and
$$\|T^t_2\|_{L^p\to L^p}\leq C_p\,(1+|t| )\,.
$$
\end{itemize}
\end{teo}
\begin{proof}
Without loss of generality we shall assume that $t>0$. We first prove (i). Let $\chi\in C^{\infty }_c(\RR)$ be an even function such that $\chi(s)=1$ if $|s|\leq 1/2$, and $\chi(s)=0$ if $|s|\geq 1$. Put $\psi_0(s):=\chi(s/2)$, and $\psi_j(s)=\chi(2^{-j-1}s)-\chi(2^{-j}s)=\psi(2^{-j}s)$, $j=1,\ldots,\infty$, where $\psi(s):=\chi(s/2)-\chi(s)$ is supported in $\{s:1/2\leq |s|\leq 2  \}$. Then $\psi_0$ is supported in $[-2,2]$, $\psi_j$ in $\{s:2^{j-1}\leq |s|\leq 2^{j+1}  \}$ for $j\geq 1$, and
\begin{equation}\label{sum}
\sum_{j=0}^{\infty}\psi_j(s)=1 \qquad \forall s\in\RR\,.
\end{equation}
We shall restrict ourselves to the case $1\leq p< 2$, since the case $p=2$ is trivial and the case $p>2$ follows from the case $p>2$ by duality. Using (\ref{sum}) we decompose the symbol $m$ as
$$m(s)=\sum_{j=0}^{\infty}m_j(2^{-j}s)\,,$$
where $m_0=m\chi$ and $m_j(s):=(m\psi_j)(2^js)=m(2^js)\,\psi(s)$, if $j\geq 1$. Notice that for all $N\in\NN$
\begin{equation}\label{normmj}
\|m_j\|_{C^N}\leq C\,2^{-\alpha j}\,,
\end{equation}
where the constant $C$ depends on the semi-norms $\|m\|_{\mathcal S^{-\alpha},k}$ only.

Then for every $f$ in $L^2(\rho)$
$$T^t_1f=\sum_{j=0}^{\infty}T_jf\qquad {\rm{in}}~L^2(\rho)\,,$$
where $T_j:=m_j\Big(\frac{\sqrt{L}}{2^j}  \Big)\,\cos\Big((2^jt)\,\frac{\sqrt{L}}{2^j}  \Big)$. Estimating the operator norms of $T_j$ on $L^1(\rho)$ by means of Proposition \ref{Wtlambda} and (\ref{normmj}) we obtain 
$$\|T_j\|_{L^1\to L^1}\lesssim\begin{cases}
2^{-\alpha j}\,(1+2^jt )^{(n-1)/2  }&{\rm{if~}}t<1\\
2^{-\alpha j}\,2^{j(n-3)/2}\,2^jt&{\rm{if~}}t\geq 1\,.
\end{cases}
$$
Interpolating the previous estimates with the trivial $L^2$ estimate $\|T_j\|_{L^2\to L^2}\lesssim2^{-\alpha j}$, we obtain the following inequalities:
\begin{equation}\label{normTj}
\|T_j\|_{L^p\to L^p}\lesssim\begin{cases}
2^{-\alpha j}\,(1+2^jt )^{(n-1)(1/p-1/2)  }&{\rm{if~}}t<1\\
2^{-\alpha j}\,2^{j(n-1)(1/p-1/2)     }\,t^{(2/p-1) }&{\rm{if~}}t\geq 1\,.
\end{cases}
\end{equation}
If $\alpha>\alpha_n(p)$, then by summation over all $j\geq 0$ (i) follows immediately.

As for (ii), observe first that if we replace $m^t_{\lambda}(s)=\psi\Big( \frac{\sqrt s}{\lambda}\Big) \,\cos(t\sqrt s)$ in Section \ref{estimatekernel} by $\tilde{m}^t_{\lambda}(s)=\psi\Big( \frac{\sqrt s}{\lambda}\Big) \,\frac{\sin(t\sqrt s)}{\sqrt s}$, then the factor $\nep^{i(R-t)s}+\nep^{i(R+t)s}$ in the corresponding kernel has to be replaced by the factor $is^{-1}\,(\nep^{i(R-t)s}-\nep^{i(R+t)s})$. By \cite[Lemma 6.2]{MT} the estimates for the function $\tilde{k}^t_{\lambda}$ and $\tilde{W}^t_{\lambda}$ are therefore the same as for ${k}^t_{\lambda}$ and ${W}^t_{\lambda}$ in Theorem \ref{Glambda} and Theorem \ref{Wtlambda}, except for an additional factor $\lambda^{-1}$. Moreover,
\begin{equation}\label{j0}
\sup_s\Big|m_j\Big( \frac{s}{2^j} \Big)\,\frac{\sin(ts)}{s}   \Big|\lesssim\begin{cases}
2^{-\alpha j}\,2^{-j}&{\rm{if~}}j\geq 1\\
t&{\rm{if~}}j=0\,.
\end{cases}
\end{equation}
Together, this implies that for $j\geq 1$, the operators $\tilde{T}_j$ which appear in the dyadic decomposition of $T^t_2$ satisfy the same estimates as $T_j$, except for an additional factor $2^{-j}$, i.e.,
\begin{align*}
\|\tilde{T}_j\|_{L^1\to L^1}&\lesssim\begin{cases}
2^{-\alpha j-j}\,(1+2^jt )^{(n-1)/2  }&{\rm{if~}}t<1\\
2^{-\alpha j}\,2^{j(n-3)/2}\,t&{\rm{if~}}t\geq 1
\end{cases}\\
&\lesssim\begin{cases}
2^{-\alpha j-j}\,2^{j(n-1)/2 }(1+t )^{(n-1)/2  }&{\rm{if~}}t<1\\
2^{-\alpha j-j}\,2^{j(n-1)/2}\,t&{\rm{if~}}t\geq 1\,.
\end{cases}
\end{align*}
By interpolating with the estimate $\|\tilde{T}_j\|_{L^2\to L^2}\lesssim2^{-\alpha j-j}$, we obtain that for $j\geq 1$
$$\|\tilde{T}_j\|_{L^p\to L^p}\lesssim\begin{cases}
2^{-\alpha j-j}\,2^{j(n-1)(1/p-1/2) }(1+t )^{(n-1)(1/p-1/2)  }&{\rm{if~}}t<1\\
2^{-\alpha j-j}\,2^{j(n-1)(1/p-1/2)}\,t^{2/p-1}&{\rm{if~}}t\geq 1\,.
\end{cases}
$$
For $j=0$ by (\ref{j0}) and Proposition \ref{Wtlambda} we obtain
$$\| \tilde{T}_0\|_{L^2\to L^2}\lesssim t\qquad \| \tilde{T}_0\|_{L^1\to L^1}\lesssim(1+t)\,,$$
hence 
$$\| \tilde{T}_0\|_{L^p\to L^p}\lesssim(1+t)^{2/p-1}\,t^{2-2/p}\lesssim\begin{cases}    
t^{2-2/p}&{\rm{if~}}t<1\\
t&{\rm{if~}}t\geq 1\,.
\end{cases}
$$
By summation over $j\geq 1$ and $j=0$ we obtain that if $\alpha>\alpha_n(p)-1$, then
\begin{align*}
\|T_2^t\|_{L^p\to L^p}&\lesssim\begin{cases}
1&{\rm{if~}}t<1\\
t^{2/p-1}+t&{\rm{if~}}t\geq 1
\end{cases}\\
&\lesssim1+|t|\,,
\end{align*}
as required.
\end{proof}
Let $u(t,x)$ be the solution of the Cauchy problem
\begin{equation}\label{Cauchy}
\partial_t^2u+Lu=0,\qquad u(0,\cdot)=f\qquad \partial_tu(0,\cdot)=g.
\end{equation}
For $f,\,g$ in $L^2(\rho)$ the solution $u$ is given by 
$u(t,\cdot)=\cos (t\sqrt{L})f+\frac{\sin(t\sqrt{L})}{\sqrt{L}}g$. We define the adapted Sobolev norms 
$$
\|\phi\|_{L^p_{\alpha}}=\| \big(1+L \big)^{\alpha/2} \,\phi\|_{L^p}\qquad \forall \alpha\in\mathbb R.
$$
From Theorem \ref{Tt1Tt2} we immediately obtain 
\begin{coro}\label{MPA}
If $1< p< \infty$, $\alpha_0>(n-1)\,\big|1/p-1/2\big|$ and $\alpha_1>(n-1)\,\big|1/p-1/2\big|-1$, then
\begin{equation}\label{Lpnorm}
\|u(t,\cdot)\|_{L^p}\leq C_p\,\big( (1+|t| )^{|\frac 2p- 1|}\, \|f\|_{L^p_{\alpha_0}} + (1+|t| )\,\|g\|_{L^p_{\alpha_1}} \big).
\end{equation}
\end{coro}
\begin{remark}
It is likely that the estimate (\ref{Lpnorm}) even holds for 
$\alpha_0=(n-1)\,\big|1/p-1/2\big|$ and $\alpha_1=(n-1)\,\big|1/p-1/2\big|-1$. This would be the counterpart to corresponding results in the Euclidean setting  by A.~Miyachi and J.C.~Peral \cite{M, P}. To prove it, we hope to prove an   endpoint result in Theorem \ref{Tt1Tt2} for $p=1.$ More precisely, we expect that if $m$ is a symbol in $\mathcal S^{-(n-1)/2}$, then the operator $T^t_1=m(\sqrt \D)\,\cos(t\sqrt \D)$ extends to a bounded operator from $H^1$ to $L^1(\rho)$, and
$$\|T^t_1\|_{H^1\to L^1}\leq C (1+|t|)\,,
$$
where $H^1$ is the ''non-standard'' Hardy space introduced in \cite{V2}. A similar results should hold also for the operator $T^t_2$. This will be the object of further investigation.  
\end{remark}
\begin{remark}
We recall that the same problem on a noncompact symmetric space $X$ of rank one for the wave equation associated with the operator $\LQ$ has been studied by A.~Ionescu \cite{I}. He proved that if $v(t,x)$ is the solution of the Cauchy problem
\begin{equation}\label{Cauchy2}
\partial_t^2v+\LQ v=0,\qquad v(0,\cdot)=f\qquad \partial_tv(0,\cdot)=g\,,
\end{equation}
for $f,\,g$ in $L^2(\lambda)$, then if $1< p< \infty$, $\alpha_0>(n-1)\,\big|1/p-1/2\big|$ and $\alpha_1>(n-1)\,\big|1/p-1/2\big|-1$
\begin{equation}\label{Lpnormionescu}
\|v(t,\cdot)\|_{L^p(\lambda)}\leq C_p\,e^{Q|1/2-1/p|t}\,\big[ \|f\|_{L^p_{\alpha_0}(\lambda)} + (1+|t| )\,\|g\|_{L^p_{\alpha_1}(\lambda)} \big]\,,
\end{equation}
where we recall that $\lambda$ denotes the left Haar measure on $S$. 
Note that the regularity indices $\alpha_0$ and $\alpha_1$ are the same as above but there is an exponential growth with respect to the time $t$. Ionescu proved also an $L^{\infty}$--$BMO$ result. 
\end{remark}



\begin{thebibliography}{CDKR1}



\bibitem[A1]{A1} F.~Astengo, 
           \emph{The maximal ideal space of a heat algebra on solvable extensions of H-type groups},
          Boll. Un. Mat. Ital. A(7){\bf{ 9}} (1995), 157--165

\bibitem[A2]{A2} F.~Astengo, 
            \emph{Multipliers for a distinguished Laplacean on solvable extensions of H-type groups},
            Monatsh. Math. {\bf 120} (1995), 179--188

\bibitem[ADY]{ADY} J.~P.~Anker, E.~Damek, C.~Yacoub,
                   \emph{Spherical Analysis on harmonic $AN$ groups},
                   Ann. Scuola Nom. Sup. Pisa Cl. Sci. (4) {\bf 23} (1996), 643--679


\bibitem[CDKR1]{CDKR1} M.~Cowling, A.~H.~Dooley, A.~Kor\'anyi, F.~Ricci,
                       \emph{H-type groups and Iwasawa dwcompositions},
                       Adv. Math. {\bf 87} (1991), 1--41

\bibitem[CDKR2]{CDKR2} M.~Cowling, A.~H.~Dooley, A.~Kor\'anyi, F.~Ricci,
                       \emph{An approach to symmetric spaces of rank one via groups of Heisenberg type},
                       J. Geom. Anal. {\bf 8} (1998), 199--237


\bibitem[CGHM]{CGHM} M.~Cowling, S.~Giulini, A.~Hulanicki, G.~Mauceri,
                     \emph{Spectral multipliers for a distinguished laplacian on                     certain groups of exponential growth}, 
                     Studia Math. {\bf 111} (1994), 103--121

\bibitem[CGM]{CGM} M.~Cowling, S.~Giulini, S.~Meda, 
\emph{Oscillatory multipliers related to the wave equation on noncompact symmetric spaces},
                 J. London Math. Soc. {\bf 66}  (2002), 691--709

\bibitem[D1]{D1} E.~Damek, 
{\emph{Curvature of a semidirect extension of a Heisenberg type nilpotent group}},
Colloq. Math. {\bf{53}} (1987), 249--253

\bibitem[D2]{D2} E.~Damek,
{\emph{Geometry of a semidirect extension of a Heisenberg type nilpotent group}}, Colloq. Math. {\bf{53}} (1987), 255--268

\bibitem[DR1]{DR1} E.~Damek, F.~Ricci,
{\emph{A class of nonsymmetric harmonic Riemannian spaces}},
Bull. Amer. Math. Soc. {\bf{27}} (1992), 139--142

\bibitem[DR2]{DR2} E.~Damek, F.~Ricci,
{\emph{Harmonic analysis on solvable extensions of $H$-type groups}},
J. Geom. Anal. {\bf{2}} (1992), 213--248


\bibitem[GM]{GM} S.~Giulini, S.~Meda, 
\emph{Oscillating multipliers on noncompact symmetric spaces},  
J. Reine Angew. Math.  {\bf 409}  (1990), 93--105

\bibitem[He]{He} W.~Hebisch, 
\emph{The subalgebra of $L^1(AN)$ generated by the Laplacian},
Proc. Amer. Math. Soc.  {\bf{117}}  (1993),  547--549

\bibitem[HS]{HS} W. Hebisch, T. Steger,
                 {\emph{Multipliers and singular integrals on expo\-nen\-tial growth groups}},
                 Math. Z. {\bf{245}} (2003), 37--61

\bibitem[H]{H} G.A.~Hunt,
{\emph{Semigroups of measures on Lie groups}},
Trans. Amer. Math. Soc. {\bf{81}} (1956), 264--293

\bibitem[I]{I} A.D.~Ionescu, 
\emph{Fourier integral operators on noncompact symmetric spaces of real rank one},  
J. Funct. Anal.  {\bf 174}  (2000), 274--300

\bibitem[K]{K} A.~Kaplan, 
               \emph{Fundamental solutions for a class of hypoelliptic PDE generated by composition of quadratic forms},
               Trans. Amer. Math. Soc. {\bf 258} (1975), 145--159

\bibitem[M]{M} A.~Miyachi, 
 \emph{On some estimates for the wave equation in $L^p$ and $H^p$},
J. Fac. Sci. Univ. Tokyo Sect. IA Math. {\bf 27} (1980), 331--354 

\bibitem[Mu]{Mu} D.~ M{\"u}ller,
 \emph{Functional calculus on {L}ie groups and wave propagation},
  Proceedings of the International Congress of Mathematicians,
  Vol. II (Berlin, 1998), number Extra Vol. II, 1998, pages 679--689.


\bibitem[MT]{MT} D.~M\"uller, C.~ Thiele, 
                 \emph{ Wave equation and multiplier estimates on $ax+b$ groups}, Studia Math.  {\bf 179}  (2007), 117--148
                 

\bibitem[P]{P} J.C.~Peral, 
              {\emph {$L^p$-estimate for the wave equation}},
              J. Funct. Anal. {\bf 36} (1980), 114--145




\bibitem[V1]{V1} M.~Vallarino, 
                 \emph{ Spectral multipliers on harmonic extensions of $H$-type groups},
                 J. Lie Theory {\bf 17} (2007), 163--189 

\bibitem[V2]{V2} M.~Vallarino,
             \emph{Spaces $H^1$ and $BMO$ on $ax+b$--groups}, 
             Preprint n.548 Dipartimento di Matematica dell'Universit\`a di Genova (2007), arXiv:0804.4615
 

                 
\end{thebibliography}
\end{document}